\documentclass[11pt]{amsart}
\usepackage{amssymb, amsmath, amsthm}
\usepackage[latin1]{inputenc}
\usepackage{graphicx}
\usepackage[final]{hyperref}
\usepackage[a4paper, centering]{geometry}
\usepackage{color}
\usepackage{graphicx} 

\newtheorem{theorem}{Theorem}
\newtheorem{proposition}[theorem]{Proposition}
\newtheorem{lemma}[theorem]{Lemma}
\newtheorem{corollary}[theorem]{Corollary}
\theoremstyle{definition}
\newtheorem{definition}[theorem]{Definition}
\theoremstyle{remark}
\newtheorem{remark}[theorem]{Remark}

\def\R{\mathbb{R}}

\def\pscal#1#2{\left\langle#1,\,#2\right\rangle}

\def\dist{d}

\def\osubjet{J^{2, -}_{\Omega}}
\def\osuperjet{J^{2, +}_{\Omega}}
\def\subjet{J^{2, -}_{\overline\Omega}}

\def \e{\varepsilon}
\def \ipo {(h\Omega)}
\def \ipu {(hu)} 

\def\csubjet{\overline{J}^{2, -}_{\overline\Omega}}

\def\X{\mathbf{X}}
\def\Xe{\X_{\epsilon}}

\DeclareMathOperator{\Cut}{\overline \Sigma}  
\DeclareMathOperator{\high}{M}
\DeclareMathOperator{\conv}{conv}
\DeclareMathOperator{\extr}{extr}
\DeclareMathOperator{\argmax}{argmax}

\begin{document}

\title[Inhomogeneous infinity Laplacian]%
{On the Dirichlet and Serrin problems for the \\ inhomogeneous infinity Laplacian in convex domains: \\ Regularity and geometric results}%
\author[G.~Crasta, I.~Fragal\`a]{Graziano Crasta,  Ilaria Fragal\`a}
\address[Graziano Crasta]{Dipartimento di Matematica ``G.\ Castelnuovo'', Univ.\ di Roma I\\
P.le A.\ Moro 2 -- 00185 Roma (Italy)}
\email{crasta@mat.uniroma1.it}

\address[Ilaria Fragal\`a]{
Dipartimento di Matematica, Politecnico\\
Piazza Leonardo da Vinci, 32 --20133 Milano (Italy)
}
\email{ilaria.fragala@polimi.it}

\keywords{}
\subjclass[2010]{Primary 49K20, Secondary 49K30, 35J70,  35N25.  }
\date{October 22, 2014; revised March 27, 2015}

\begin{abstract}  Given an open bounded subset $\Omega$ of $\mathbb{R}^n$, which is convex and satisfies an interior sphere condition,
we consider the pde $-\Delta_{\infty} u = 1$ in $\Omega$,  subject to the homogeneous boundary condition $u = 0$ on $\partial \Omega$. We prove that the unique solution to this Dirichlet problem is power-concave (precisely, 3/4 concave) and it is of class $C ^1(\Omega)$. We then investigate the overdetermined Serrin-type problem, formerly considered in \cite{butkaw}, obtained by adding the extra boundary condition $|\nabla u| = a$ on $\partial \Omega$; by using a suitable $P$-function we prove that, 
if $\Omega$  satisfies the same assumptions as above and in addition contains a ball with touches $\partial \Omega$ at two diametral points, then the existence of a solution to this Serrin-type problem implies that necessarily the cut locus and the high ridge of $\Omega$ coincide. In turn, in dimension $n=2$, this entails that $\Omega$ must be a stadium-like domain, and in particular it must be a ball in case its boundary is of class $C^2$.  
\end{abstract}

\maketitle

\section{Introduction}

\subsection{Setting of the problem.}
The infinity Laplacian is the differential operator defined for smooth functions $u$ by  
$$\Delta _\infty u := \nabla ^ 2 u \nabla u \cdot \nabla u\,.$$ 
It was firstly discovered by Aronsson in the sixties \cite{Aro}, and afterwards a fundamental contribution came by Jensen \cite{Jen}, 
who proved the well-posedness of the Dirichlet problem
$\Delta _\infty u = 0$ in $\Omega$ with $u = g$ on $\partial \Omega$, for every boundary datum 
$g \in C  (\partial \Omega)$.  
(Here and in general when dealing with the infinity Laplacian, solutions must be intended in the viscosity sense, as the operator is not in divergence form). Moreover, Jensen  proved that $u$ is
characterized by the variational property of being a so-called {\it absolute minimizing Lipschitz extension of $g$}, meaning that it minimizes the $L ^ \infty$ norm of the gradient on every set $A \subset \subset \Omega$, among all functions which have the same trace  on $\partial A$. 
In particular, this property justifies the name ``infinity Laplacian''; a general existence theory of  Calculus of Variations in the sup-norm and related Aronsson--Euler type equations has been later developed by Barron, Jensen and Wang \cite{BaJeWa}. 

An excellent paper  reviewing of the state of the art on problems involving the infinity-Laplacian up to 2004 
is \cite{ArCrJu}.   
In the last decade these problems have raised an increasing interest in the pde community, 
stimulated also by their connections with tug-of-war games (see e.g.\ \cite{PSSW}), and 
further progresses have been made in both existence and regularity theory. 

Concerning advances in the existence theory, a notable contribution has to be ascribed to Lu and Wang, who proved in particular the well-posedness 
of the Dirichlet problem
\begin{equation}\label{f:dirich}
\begin{cases} 
-\Delta_{\infty} u = 1 &\text{in}\ \Omega\\
u = 0 &\text{on}\ \partial\Omega\,, 
\end{cases}
\end{equation}
see  \cite{LuWang}, where the Authors deal also with the case of non-constant source terms with constant sign; more general source terms have been recently considered  in  \cite{BhMo}. 

Concerning regularity matters,  the mostly investigated case is the one of infinity harmonic functions:
they
have been proved to be  differentiable in any space dimension $n$  by Evans and Smart \cite{EvSm}, whereas their $C ^ {1, \alpha}$ regularity 
(which is the optimal one expected) has been proved only for  $n=2$ by Evans and Savin \cite{EvSav}, and remains a major open problem in higher dimensions. 
Recently, the everywhere differentiability property in any dimension $n$ has been extended by Lindgren  \cite{Lind} to a class of inhomogeneous Dirichlet problems including \eqref{f:dirich}
(see also \cite{SWZ} for the same kind of result for some  Aronsson-type equations).  
At present,
no $C^1$ type regularity result is available to the best of our knowledge for the  Dirichlet problem \eqref{f:dirich}.

A new investigation direction in this field has been suggested by Buttazzo and Kawohl in the pioneering paper \cite{butkaw}, where they started the study of the following overdetermined problem: 
\begin{equation}
\label{f:serrin}
\begin{cases}
-\Delta_{\infty} u = 1 &\text{in}\ \Omega\\
u = 0 &\text{on}\ \partial\Omega
\\
|\nabla u| = a &\text{on}\ \partial\Omega\,. 
\end{cases}
\end{equation}

The analogous problem with the classical Laplacian in place of the infinity Laplacian was studied by 
Serrin, who proved the seminal symmetry result stating that  existence of a solution implies that $\Omega$ is a ball \cite{Se}. For its mathematical beauty and the elegance of its proof, which is based on the moving planes method by Alexandrov, Serrin result has become a masterpiece in pde's. It has originated a huge amount of literature, including alternative proofs and many generalizations, about the cases when the Laplacian is replaced by a possibly degenerate elliptic operator and when the elliptic problem is stated on an exterior domain, or on a ring-shaped domain, or on a domain with not smooth boundary. Since it is impossible to give here an exhaustive bibliography on overdetermined boundary value problems, we limit ourselves to quote the papers \cite{BH, BrPr,  f, fg, fgk, GL, K2, Vog}, where many further relevant references can be found. 

Now what happens for problem \eqref{f:serrin} is that {\it all} the methods known in the literature to deal with overdetermined boundary value problems completely fail.
There are several deep reasons which may be addressed  for this fact, among which the high degeneracy of the operator, the failure of a strong maximum principle and the lack of regularity results for solutions to the Dirichlet problem. 
Actually, until now only a highly simplified version of problem \eqref{f:serrin} has been successfully investigated: it
consists in studying for which domains  $\Omega$ the unique solution $\overline u$ to the Dirichlet problem 
\eqref{f:dirich}
depends only on the distance $d _{\partial \Omega}$ from the boundary of $ \Omega$. Functions depending only on $d _{\partial \Omega}$ are called {\it web functions}, since if $\Omega$ is a polygon their level lines look like a spider web; for a short history of web functions, and an example of their application in variational problems, see \cite{CFGa}. 
If $\overline u$ is a web function, it has a constant normal derivative on $\partial \Omega$, and hence it solves \eqref{f:serrin}; on the other hand, it is clear that problem \eqref{f:serrin} might well have solutions which are {\it not} web functions. 

A  necessary and sufficient  condition for $\overline u$ being  a web function 
is the concidence between the cut locus and the high ridge of $\Omega$  (for their definition see the end of this Introduction). 
Let us emphasize that this geometric phenomenon
has been firstly discovered by 
Buttazzo and Kawohl in \cite{butkaw}. 
Afterwards, the same result has been proved in \cite{CFc} under milder regularity assumptions. 
A complete characterization of sets satisfying such geometric condition in two space dimensions has been provided  in \cite{CFb}, as parallel neighborhoods of $C ^ {1,1}$ one-dimensional manifolds
(in particular, they do not need to be balls, unless they are asked in addition to be simply connected and of class $C ^2$). 

This paper can be framed into the above described state of the art
(see also \cite{CFdproc} for a review), and deals with the following two mutually related topics:

\smallskip
(i) {\it About the  Dirichlet problem \eqref{f:dirich}}:
Does its unique solution  enjoy stronger regularity than everywhere differentiability? This question is relevant in connection with the study of the overdetermined problem \eqref{f:serrin}, but it has an autonomous interest 
since, as mentioned above, even the $C^1$ regularity of infinity harmonic functions is still object of investigation in dimension higher than $2$.

\smallskip
(ii) {\it About the Serrin-type problem \eqref{f:serrin}}:
If one does not work within the restricted class of web-functions, which kind of geometric information on $\Omega$ can be inferred from the existence of a solution? 
Is the coincidence of cut locus and high ridge still a necessary condition? 
The fact that one can no longer reduce the problem to an ODE for a function depending on $d _{\partial \Omega}$ increases dramatically the difficulty level, and some completely new approach is needed with respect to the methods employed in \cite{butkaw} and \cite{CFc}.  

\smallskip
\subsection{Outline of the results.}
As a  first step, in Section \ref{secconc} we prove a power-concavity result for the solution $\overline u$ to problem \eqref{f:dirich}: precisely we prove that it is $3/4$ concave, provided the domain $\Omega$ is convex and satisfies an interior sphere condition (see Theorem \ref{t:34}). 
This result, which is obtained by the convex envelope method introduced by Alvarez, Lasry and Lions in \cite{ALL}, yields as a crucial by-product that, under the same assumptions on $\Omega$, the solution $\overline u$ is locally semiconcave  (see Corollary \ref{c:locsemiconc}). 
We remark that a similar power-concavity property has been proved by Sakaguchi in \cite{Sak}
in the case of the $p$-Laplace operator.

In Section \ref{secdiffe} we exploit the local semiconcavity of the solution $\overline u$ in order to obtain its $C ^ 1$ regularity (see Theorem \ref{t:diff}). Incidentally, we provide an alternative  proof of the differentiability of $\overline u$ which works in convex domains and is completely different   with respect to the one given in \cite{Lind}. Actually, 
the main ingredient of our approach is a new estimate holding for locally semiconcave functions near singular points (see Theorem \ref{t:estid}): 
we use this estimate within  a contradiction argument, in order to construct ad hoc viscosity test functions
for problem \eqref{f:dirich}, which allows to conclude that $\overline u$ cannot have singular points. 

In Section \ref{secP} we introduce the $P$-function given by
$$P (x) := \frac{|\nabla \overline u| ^ 4 }{4} + \overline u\,.$$
Note that, thanks to the $C ^1$ regularity result obtained for $\overline u$, the function $P$ is continuous in $\Omega$.
The idea is that the existence of a solution to the overdetermined problem \eqref{f:serrin} 
(or equivalently the constancy of $|\nabla \overline u |$ over the boundary) might imply that $P$ is constant on the whole of $\Omega$. Actually,  
should the function $P$ be constant on the whole of $\Omega$, one would obtain immediately the information that cut locus and high ridge of $\Omega$ coincide (see Proposition \ref{p:P1}). 
In order to investigate the possible constancy of $P$, we study its behaviour along the steepest ascent lines of $\overline u$, intended as trajectories of the Cauchy problem $$
\begin{cases}
\dot \gamma (t) = \nabla \overline u (\gamma (t)) 
\\
\gamma (0) = x \in \overline \Omega\,.& 
\end{cases}
$$
Indeed, it is easy to see that the map $t \mapsto P (\gamma (t))$ has vanishing first order derivative almost everywhere (see Lemma \ref{l:easy}) and that, should the function $P$ be constant along a trajectory $\gamma$, one could  immediately compute the solution along it (see Proposition \ref{p:P2}). 

Unfortunately, the constancy of $P$ along a trajectory cannot be deduced from the vanishing property of the first derivative, because $\overline u$ is not known to be $C ^ {1,1}$ (and actually it cannot expected to be so, see below), 
so that the map $P \circ \gamma$ is {\it not} absolutely continuous. 

Nevertheless, we manage to get some control on the properties of trajectories, and to infer some
information on the {\it global} behavior of $P$ on $\Omega$. The approach 
we adopt consists in constructing the unique forward gradient flow associated with  $\overline u$ (which can be done
thanks to its semiconcavity, see Lemma \ref{l:geo}), and then approximating it by the sequence of gradient flows associated with  the supremum convolutions of $\overline u$ (which enjoy
$C ^ {1,1}$ regularity and, by the ``magical properties" of their superjets,  turn out to be sub-solutions to the pde, 
 see Lemma \ref{l:approx1}).  By this way, in Theorem \ref{t:ineqP}, we obtain the crucial estimates 
$$\min _{\partial \Omega} \frac{|\nabla \overline u| ^ 4 }{4} \leq P (x)  \leq \max _{\overline {\Omega}} \overline u
\qquad \forall  x \in \overline \Omega \,.
$$
These bounds can be used to infer  some information both on the geometry of domains on which problem \eqref{f:serrin} admits a solution, and on the regularity of the solution to problem \eqref{f:dirich}. 
This is done respectively in the last two sections of the paper. 
In Section \ref{secgeo} we prove  that the existence of a solution to problem \eqref{f:serrin} entails the coincidence of cut locus and high ridge provided the domain $\Omega$ is convex and contains an inner ball which touches $\partial \Omega$ at two diametral points (see Theorem \ref{t:serrin2}). For instance, this excludes existence of a solution to problem \eqref{f:serrin} when $\Omega$ is an ellipse. It is our belief that that both the convexity and the ``diametral touching ball'' conditions are not necessary for the validity of the result, but by now proving it in full generality remains an open problem. 
We wish to emphasize that combining Theorem \ref{t:serrin2} with the results proved in our previous paper \cite{CFb} reveals 
an interesting phenomenon which seems to be completely new in the field of overdetermined problems, and more generally in the interplay between geometry and pde's:  convex domains where problem \eqref{f:serrin}  
admits a solution may obey or not symmetry according to the regularity of their boundary; more precisely, in dimension $n=2$, they must be spherical as soon as they are of class $C ^2$, but may be nonspherical (precisely stadium-like domains) if they do not enjoy such regularity (see Corollary \ref{corgeo}). 
This seems somehow to reflect the fact that regularity properties for the solution to the pde finer than $C ^ 1$ are a delicate stuff.  
Such properties are discussed in the final
Section \ref{secreg} where, via the use of the $P$-function, 
we show that the expected optimal regularity of $\overline u$ is $C ^ {1, \alpha}$ with $\alpha \leq 1/3$
(for the precise statements see Propositions \ref{p:notreg1} and \ref{p:notreg2}).

\subsection {Some preliminary notions.} 

Let us specify  what we mean by a {\it solution} to problems \eqref{f:dirich} and  \eqref{f:serrin}. 
For convenience of the reader, let us first remind the definition of viscosity sub- and super-solutions.
Recall first that second order sub-jet (resp.\ super-jet),
$\osubjet u (x_0)$ (resp. $\osuperjet u (x_0)$), of a function $u\in C(\overline{\Omega})$
at a point $x_0\in {\Omega}$,  is by definition the set of pairs
$(p, A) \in \R ^n \times \R ^ { n \times n }_{{\rm sym}}$ such that
\[ u (y) \geq  (\leq) \ u ( x_0) + \pscal{ p}{y- x_0} 
+ \frac{1}{2} \pscal {A (y- x_0)}{y- x_0} + o ( |y - x_0|^2) 
\qquad \hbox{ as } 
 y \to x_0,\ y\in {\Omega}\,,
\]
Then,  following \cite{CHL},  
a viscosity subsolution  to the equation $-\Delta_\infty u -1 = 0$
is a function
$u\in C({\Omega})$  which, for every $x_0 \in \Omega$,  satisfies
\begin{equation}\label{f:subsol1}
-\Delta_\infty\varphi(x_0)-1\leq 0\
\quad
\text{whenever}\
\varphi\in C^2(\Omega)\
\text{and $u - \varphi$ has a local maximum at $x_0$},
\end{equation}
or equivalently
\begin{equation}\label{f:subsol2}
- \langle X p, p \rangle -1\leq 0\
\quad
\forall (p, X) \in J ^ { 2 , +} _\Omega u (x_0)\, .
\end{equation}
Similarly, a viscosity super-solution to the equation $-\Delta_\infty u -1 = 0$
is a function
$u\in C({\Omega})$  which, for every $x_0 \in \Omega$,  satisfies
\begin{equation}\label{f:supersol1}
-\Delta_\infty\varphi(x_0)-1\geq 0\
\quad
\text{whenever}\
\varphi\in C^2(\Omega)\
\text{and $u- \varphi$ has a local minimum at $x_0$},
\end{equation}
or equivalently
\begin{equation}\label{f:supersol2}
- \langle X p, p \rangle -1\geq 0\
\quad
\forall (p, X) \in J ^ { 2 , -} _\Omega u (x_0)\, .
\end{equation}

By a viscosity solution to the equation $-\Delta_\infty u -1 = 0$ 
we mean a function $u\in C(\overline{\Omega})$
which is both a viscosity sub-solution and a viscosity super-solution on $\Omega$.

By a solution to problem (\ref{f:dirich}), we mean a function $u\in C(\overline{\Omega})$ such that
$u=0$ on $\partial\Omega$ and $u$ is a viscosity solution
to $-\Delta_\infty u = 1$ in $\Omega$.

By saying that a the overdetermined boundary value problem (\ref{f:serrin}) admits a solution, we mean that 
the following regularity hypothesis is fulfilled
\begin{itemize}
\item[$\ipu$]\ \ \  the unique viscosity solution $u$
to problem \eqref{f:dirich}
satisfies 
$$
\exists \, \delta>0 \ :\ u \text{ is of class } C ^1 \text{ on } \{ x \in \overline \Omega \ :\ {\rm dist} (x, \partial \Omega) < \delta\}\,,
$$
\end{itemize}
and that $|\nabla u| = a$ on $\partial \Omega$.

Finally, let us introduce some definitions related to the distance function to the boundary of $\Omega$, which will be denoted by $d _{\partial \Omega}$. We let $\Sigma (\Omega)$ be the set of points in $\Omega$ where $d_{\partial \Omega}$ is not differentiable, and we call {\it cut locus} and {\it high ridge}
the sets given respectively by
\begin{eqnarray} 
\hbox{$\Cut(\Omega)$ := the closure of   $\Sigma (\Omega)$ in $\overline \Omega$
} \qquad  \qquad& \label{cut} \\ \noalign{\smallskip}
\hbox{$\high (\Omega)$ := the set  where $d _{\partial \Omega}(x) = \rho _\Omega:= \max _{ \overline \Omega} d _{\partial \Omega}\,. $ } & \label{high}
\end{eqnarray}
Moreover, we denote by $\phi _\Omega$ the web-function defined on $\Omega$ by 
\begin{equation}\label{defphi} 
\phi _\Omega (x) : = c_0 \left[\rho_\Omega ^{4/3} - (\rho_\Omega - \dist_{\partial \Omega}(x))^{4/3}\right] \,, \qquad \hbox{ where } c_0 := 3^{4/3} / 4\,.
\,  
\end{equation}

\section{Power-concavity and semiconcavity of solutions}\label{secconc}

Throughout the paper, $\Omega$ denotes a nonempty open bounded subset of $\R ^n$. 

Most of our results will be proved under the following 
additional hypothesis (which however will be specified in each statement):

\begin{itemize}
\item[$\ipo$]  \ \ \
$\Omega$ is  convex  and 
satisfies an interior sphere condition. 
\end{itemize}

\begin{theorem}\label{t:34}
Assume $\ipo$, and let $u$ be the solution to problem $(\ref{f:dirich})$. 
Then $u ^ {3/4}$ is concave in $\Omega$. \end{theorem}

Before proving Theorem \ref{t:34}, we observe that it readily implies the following semiconcavity result. 
We recall that  $u: \Omega \to \R$ is called {\it semiconcave (with constant $C$) in $\Omega$} if 
\[
u(\lambda x + (1-\lambda)y) \geq \lambda u(x) + (1-\lambda) u(y)
-C\frac{\lambda(1-\lambda)}{2}\, |x-y|^2
\qquad \forall [x,y] \subset \Omega \ \text { and } \ \forall \lambda \in [0,1]\,.
\]
We say that $u$ is {\it locally semiconcave in $\Omega$} if it is semiconcave on compact subsets of $\Omega$. 

\begin{corollary}\label{c:locsemiconc}
Assume $\ipo$, and let $u$ be the solution to  problem $(\ref{f:dirich})$. 
Then $u$ is locally semiconcave in $\Omega$.
\end{corollary}

\proof Given $\epsilon > 0$, we claim that $u$
is semiconcave with constant $C_{\epsilon} := 4 \epsilon^{-1/2} M_\epsilon^2 / 9$
in the set $U_{\epsilon} := \{x\in\Omega:\ u(x) \geq\epsilon\}$,
where 
$M_\epsilon$ is the Lipschitz constant of $w:= u ^ {3/4}$ on the
compact set $U_\epsilon$.
Namely,
the function
$\psi(t) := t^{4/3} - 2 \epsilon^{-1/2} t^2 / 9$ is
concave in $[\epsilon^{3/4}, +\infty)$. 
Then the inequality  $\psi \big  ( \lambda w(x) + (1- \lambda) w ( y) \big ) \geq \lambda \psi (w(x)) + ( 1- \lambda) \psi ( w (y))$ entails
\[
w(\lambda x + (1-\lambda)y)^{4/3}
\geq \lambda w(x)^{4/3} + (1-\lambda)w(y)^{4/3}
-\frac{2\epsilon^{-1/2}}{9}\, \lambda(1-\lambda)
|w(x)-w(y)|^2
\]
for every $[x,y]\subset U_{\epsilon}$ and $\lambda\in [0,1]$.
On the other hand, 
we have 
$|w(x) - w(y)| \leq M_\epsilon |x-y|$, hence we obtain
\[
u(\lambda x + (1-\lambda)y) \geq \lambda u(x) + (1-\lambda) u(y)
-C_\epsilon\frac{\lambda(1-\lambda)}{2}\, |x-y|^2,
\]
i.e., $u$ is semiconcave with semiconcavity constant $C_\epsilon$
in $U_\epsilon$.
\qed 

\bigskip

The remaining of this section is devoted to the proof of Theorem \ref{t:34}.  We start with an elementary observation which will be exploited several times throughout the paper.

\begin{remark}\label{r:pos}
The viscosity solution $u$ to problem \eqref{f:dirich} is strictly positive in $\Omega$. 
Indeed, it is nonnegative by the comparison result proved in  
 \cite[Thm.\ 3]{LuWang}. Assume by contradiction that $u (x_0) = 0$ at some point $x _0 \in \Omega$. Then 
 the function $\varphi \equiv 0$ touches $u$ from below at $x_0$, and hence
$u$ cannot be a viscosity supersolution to the equation $- \Delta 
_{\infty}  u =1$ at $x_0$.  
\end{remark}

If $u$ is the solution to \eqref{f:dirich}, 
for every $\alpha \in (0,1)$ the function
$w := -u^{\alpha}$ (which is strictly negative in $\Omega$ by Remark \ref{r:pos})
is a viscosity solution of
\[
\begin{cases}
-\Delta_\infty w - \frac{1-\alpha}{\alpha}\cdot \frac{1}{w}\, |\nabla w|^4
+ \alpha^3 (-w)^{3-3/\alpha} = 0
&\text{in}\ \Omega,
\\
w = 0
&\text{on}\ \partial\Omega.
\end{cases}
\]
In order to prove Theorem \ref{t:34}, we are going to choose $\alpha = 3/4$ and show that, if $w$ is a viscosity solution to  
\begin{equation}
\label{f:probmod}
\begin{cases}
- \Delta _ \infty w - \frac{1}{w} \Big [\frac{1}{3} |\nabla w| ^ 4 + \Big ( \frac{3}{4} \Big ) ^ 3 \Big ] = 0
&\text{in}\ \Omega,
\\
w = 0
&\text{on}\ \partial\Omega,
\end{cases}
\end{equation}
then $w$ is convex. 
To that aim, we adopt the convex envelope method introduced by Alvarez, Lasry and Lions in \cite{ALL}.  Following their notation, we denote by $w_{**}$  the largest convex function below $w$. We first show that, for every $x \in \Omega$, in the characterization
\[
w_{**} (x)  = \inf \left\{ 
 \sum _ {i = 1} ^ k \lambda _i w (x_i) \ :\ x = \sum _{i=1} ^k \lambda _ i x _i \, ,\ 
 x_i \in \overline \Omega\, , \ \lambda _i >0\, ,\ \sum  _ {i = 1} ^ k \lambda _i = 1\, , \ k \leq n+1 \right\}\,
\] 
the infimum can be attained only at interior points $x_i \in \Omega$:

\begin{lemma}\label{l:emptyjet}
Assume $\ipo$, and let $u$ be the  solution to problem \eqref{f:dirich}. Set
 $w := - u^{3/4}$.
For a fixed $x\in\Omega$, let $x_1,\ldots,x_k\in\overline{\Omega}$,
$\lambda_1,\ldots,\lambda_k > 0$, with  $\sum_{i=1} ^ k \lambda _ i = 1$, be such that
\[
x= \sum_{i=1}^k \lambda_i x_i\, , \quad
w_{**}(x) = \sum_{i=1}^k \lambda_i w(x_i).
\] 
Then $x_1, \ldots, x_k\in\Omega$.
\end{lemma}

\begin{proof}
Assume by contradiction that at least one of the $x_i$'s, say $x_1$,
belongs to $\partial\Omega$.
Let $B _ R (y)\subset \Omega$ be a ball such that $\partial B _ R (y) \cap \partial \Omega = \{ x_1 \}$.
Since $-\Delta_\infty u = 1$, by Lemma 2.2 in \cite{CEG} the function $\tilde u := -u$ enjoys the property of comparison with cones from above according to Definition 2.3 in the same paper. 
Then, by Lemma 2.4 in \cite{CEG}, the function
\[
r \mapsto \max _{x \in \partial B _ r (y)} \frac{\tilde u (x) -  \tilde u  (y)}{r} = - \min  _{x \in \partial B _ r (y)}  \frac{ u (x) -  u  (y)}{r}
\]
is monotone nondecreasing on the interval $(0, R)$. 
Namely, for all $r \in (0, R)$, there holds
\begin{equation}
\label{f:cones2}
\min_{x \in \partial B _ r (y)} \frac{ u (x) -  u  (y)}{|x-y|}  \geq \min_{x \in \partial B _ R (y)} \frac{ u (x) -  u  (y)}{|x-y|} = -  \frac{u (y)} {R}\ ,
\end{equation}
where the last equality comes from the fact that
$u$ is non-negative in $\Omega$ (cf.\ Remark \ref{r:pos}).  
By \eqref{f:cones2}, we have
\[
u (x) \geq u (y ) \Big ( 1 - \frac{|x-y|}{R} \Big ) \qquad \forall x \in B _R (y) \, ,
\]
and hence
\begin{equation}
\label{f:upbd}
w (x) \leq w (y ) \Big ( 1 - \frac{|x-y|}{R} \Big ) ^ {3/4} \qquad \forall x \in B _R (y) \, .
\end{equation}
Let us define the unit vector $\zeta := (x-x_1) / |x-x_1|$
and let $\nu = (y-x_1) / |y-x_1|$ denote the inner normal of 
$\partial\Omega$ at $x_1$.
Since $\Omega$ is a convex set and $x\in\Omega$, we have
that $\pscal{\zeta}{\nu} > 0$
and $x_1 + t\zeta\in B_R(y)$ for $t>0$ small enough.
Moreover, $w_{**}$ is affine on $[x_1, x]$: indeed, since the epigraph of $w_{**}$ is the convex envelope of the epigraph of $w$, it is readily seen that $w_{**}$ is affine on the whole set of convex combinations of the points $\{x_1, \dots, x_k \}$. 
Taking into account that $w_{**}(x_1) = w(x_1) = 0$, we infer that
there exists $\mu > 0$ such that
\[
w(x_1 + t\zeta) \geq w_{**}(x_1+t \zeta) = - \mu t
\qquad\forall t\in [0,1].
\] 
{}From \eqref{f:upbd} we obtain
\[
-\mu t \leq w(y) \left( 1 - \frac{|t\zeta - R\nu|}{R}\right)^{3/4}
= w(y) \left(\pscal{\zeta}{\nu} \frac{t}{R} + o(t)\right)^{3/4},
\qquad t\to 0^+,
\]
and, recalling that $w(y) < 0$, 
\[
\mu t^{1/4} \geq K + o(1),\qquad t\to 0^+
\]
with $K>0$, a contradiction.
\end{proof}

\begin{remark}
As a consequence of \eqref{f:upbd},
taking $x = x_1 + \lambda \nu$, 
and recalling that $w (x_1) = 0$,
it is readily seen that
\[
\lim _{\lambda \to 0 ^ +} 
\frac{w ( x_1 + \lambda \nu) - w (x_1)}{ \lambda} 
\leq  \lim _{\lambda \to 0 ^ +} \frac{w ( y)}{ \lambda} \Big ( \frac{\lambda}{R} \Big ) ^ {3/4} = - \infty\,,
\]
i.e.\ the normal derivative of $w$ with respect to
the external normal is $+\infty$ at every boundary point of $\Omega$.
This is the reason why the semiconcavity property of $u$ is stated just {\it locally} in $\Omega$ 
and not up to the boundary, cf.\ the proof of Corollary \ref{c:locsemiconc}.

\end{remark}

\bigskip
On the basis of the lemma just proved, we can now establish that the convex envelope of a super-solution to 
\eqref{f:probmod} is still a super-solution. 

\begin{proposition}\label{p:All}
Assume $\ipo$. If $w$ is a viscosity super-solution to \eqref{f:probmod}, 
then also $w_{**}$ is a viscosity 
super-solution to the same problem.
\end{proposition}

\begin{proof}
Let $x\in\Omega$  and consider $(p,A)\in\subjet w_{**}(x)$.
Recall that  the second order sub-jet 
$\subjet v (x_0)$ of a function $v\in C (\overline{\Omega})$
at a point $x_0\in\overline{\Omega}$ is by definition the set of pairs
$(p, A) \in \R ^n \times \R ^ { n \times n }_{{\rm sym}}$ such that
\[
v (y) \geq v ( x_0) + \pscal{ p}{y- x_0} 
+ \frac{1}{2} \pscal {A (y- x_0)}{y- x_0} + o ( |y - x_0|^2) 
\qquad \hbox{ as } 
 y \to x_0,\ y\in\overline{\Omega}\,,
\]
whereas its ``closure'' $\csubjet v (x_0)$
is the set of $(p,A)\in \R ^n \times \R ^ { n \times n }_{{\rm sym}}$
for which there is a sequence $(p_j, A_j)\in \subjet v(x_j)$
such that
$(x_j, v(x_j), p_j, A_j) \to (x_0, v(x_0), p, A)$.

For every $\epsilon>0$ small enough,
applying Proposition 1 in \cite{ALL}
and Lemma~\ref{l:emptyjet},
we obtain
points $x_1,\ldots, x_k \in \Omega$,
positive numbers $\lambda_1,\ldots,\lambda_k$ satisfying
$\sum_{i=1}^k \lambda_i = 1$,
and elements
$(p, A_i) \in \csubjet w(x_i)$, with 
$A_i$ positive semidefinite,
such that
\[
\sum_{i=1}^k \lambda_i x_i = x,\quad
\sum_{i=1}^k \lambda_i w(x_i) = w_{**}(x),\quad
A-\epsilon A^2 \leq 
\left(\sum_{i=1}^k \lambda_i A_i^{-1}\right)^{-1}.
\]
We recall that, here and in the sequel,
it is not restrictive to assume that 
the matrices $A$, $A_1,\ldots, A_k$ are positive definite, 
since the case of degenerate matrices can be handled
as in \cite{ALL}, p.~273.

Set for brevity $F(w, p, Q) := - {\rm tr} 
\big (  (p \otimes p) Q  \big )  - \frac{1}{w}  
\Big (\frac{1}{3} |p| ^ 4 + c \Big )$, with $c:= \Big ( \frac{3}{4} \Big ) ^ 3$. 
Since $w$ is a super-solution to \eqref{f:probmod}, we have $F(w(x_i), p, A_i) \geq 0$, i.e.
\[
-w(x_i) \leq \frac{1}{\pscal{A_i p}{p}}\left(\frac{1}{3}|p|^4 + c\right)\, , 
\]
so that
\[
- \frac{1}{\sum_{i=1}^k \lambda_i w(x_i)}
\left(\frac{1}{3}|p|^4 + c\right)
\geq
\left(\sum_{i=1}^k \lambda_i \frac{1}{\pscal{A_i p}{p}}\right)^{-1}.
\]

\smallskip
Then, using the degenerate ellipticity of $F$,  we obtain
\[
\begin{split}
F(w_{**}(x), p, A-\epsilon A^2) & \geq
-\pscal{\left(\sum_{i=1}^k \lambda_i A_i^{-1}\right)^{-1} p}{p}
- \frac{1}{\sum_{i=1}^k \lambda_i w(x_i)}
\left(\frac{1}{3}|p|^4 + c\right)
\\ & \geq
-\pscal{\left(\sum_{i=1}^k \lambda_i A_i^{-1}\right)^{-1} p}{p}
+ \left(\sum_{i=1}^k \lambda_i \frac{1}{\pscal{A_i p}{p}}\right)^{-1}
\\ & \geq 0,
\end{split}
\]
where the last inequality follows from the concavity of the map
$Q\mapsto 1/ \text{tr} \big ( (p \otimes p )  Q^{-1} \big )$ proved in
\cite{ALL}, p. 286.
\end{proof}

\bigskip

Finally,  Theorem \ref{t:34} follows from Proposition \ref{p:All} 
by invoking a comparison principle: 

\bigskip
{\bf Proof of Theorem \ref{t:34}}. Let $u$ be the  solution to problem $(\ref{f:dirich})$, and let  $w= - u ^ {3/4}$.  
Then $w$ is a viscosity solution to \eqref{f:probmod} and, by Proposition \ref{p:All}, $w_{**}$ is a viscosity super-solution to the same problem, which agrees with $w= 0$ on $\partial \Omega$. By the comparison principle holding for problem (\ref{f:probmod}), we infer that $w _{**} \geq w$ in $\Omega$. (Let us point out that the validity of the comparison principle for problem (\ref{f:probmod}) can be readily deduced from the validity of the comparison principle for problem (\ref{f:dirich}) established in \cite[Thm.\ 3]{LuWang}, combined with the observation that the map $u \mapsto w = - u ^ {3/4}$ is a bijection between viscosity sub- or super-solutions to problems (\ref{f:dirich}) and sub- or super-solutions $w$ to problem (\ref{f:probmod})).  On the other hand, by definition, there holds $w _{**} \leq w$ in $\Omega$. We conclude that $w _{**} = w$ in $\Omega$, namely $w$ is convex. \qed

\begin{remark} An interesting question is whether is it possible to extend Theorem \ref{t:34}, and more generally the results of this paper, to the case of the so-called normalized infinity Laplace operator  $\Delta _\infty ^N$ (for its definition, see for instance \cite{LW10}).  
Actually, problem~\eqref{f:dirich}   with $\Delta _\infty ^N$ in place of $\Delta _\infty$ is known to have a unique solution $u$, and the natural conjecture is that $u$ is $(1/2)$-concave, 
since one can readily check that the function $w := - u ^ {1/2}$ solves
$- \Delta _\infty w = \frac{|\nabla w| ^2}{2w} (2|\nabla w| ^ 2 + 1)$.  
However, a careful inspection of the above proof of Theorem~\ref{t:34} reveals that the Alvarez-Lasry-Lions method  does not extend straightforward to this situation, in particular because the bijection exploited in the last part of the proof fails. 
This is one of the reasons why we believe that 
the case of the normalized infinity Laplace operator cannot be handled merely as a parallel variant of the infinity Laplace operator. 
We consider it as a significant direction  to be explored. 
\end{remark}


\section{$C^1$ regularity  of solutions}\label{secdiffe}

In this section we deal with the $C ^1$ regularity of the unique solution to problem \eqref{f:dirich}. 
Our strategy is as follows. 
As a first step, we prove a new estimate for locally semiconcave functions near singular points.  Then we use such estimate as a crucial tool in order
to construct suitable viscosity test functions, which prevent $u$ from being a solution to the pde at singular points. Finally we exploit the local semiconcavity result obtained in the previous section to conclude that $u$ is continuously differentiable. 

Given a function  $u\in C(\Omega)$, we denote by $\Sigma (u)$ the singular set of $u$, namely the set of points where $u$ is not differentiable. 

We recall that the {\it Fr\'echet super-differential} of 
$u$ at a point $x_0\in\Omega$ is defined by
\[
D^+ u (x_0) := 
\left\{ p \in \R ^n \ :\ \limsup _{ x \to x_0} 
\frac{u(x) - u(x_0) -\langle p, x-x_0 \rangle }{|x-x_0| }  
\leq 0  \right\} \,.
\]

We point out that, for every $x_0\in\Sigma(u)$, $D^+ u(x_0)$ is 
nonempty compact convex set which is not
a singleton; in particular, $D^+ u(x_0)\setminus \extr D^+ u(x_0)$
is not empty and contains non-zero elements. 

Then the result reads:

\begin{theorem}\label{t:estid}
Let $u\colon\Omega\to\R$ be a locally semiconcave function,
let $x_0 \in \Sigma(u)$, and let $p\in D^+ u(x_0)\setminus \extr D^+ u(x_0)$. 
Let 
$R>0$ be such that
$\overline{B}_R(x_0)\subset\Omega$, and let $C$ denote
the semiconcavity constant of $u$ on 
$\overline{B}_R(x_0)$.
Then  
there exist
a constant $K>0$ and a unit vector
$\zeta\in\R^n$ 
satisfying the following property:
\begin{equation}\label{f:estidist0}
u(x) \leq u(x_0) + \pscal{p}{x-x_0}
- K\, |\pscal{\zeta}{x-x_0}| +\frac{C}{2}
|x-x_0|^2
\quad \forall x\in\overline{B}_R(x_0)\,.
\end{equation}
In particular, for every $c>0$, setting  $\delta := \min\{K/c, R\}$, 
it holds
\begin{equation}\label{f:estidist}
u(x) \leq u(x_0) + \pscal{p}{x-x_0}
- c \pscal{\zeta}{x-x_0}^2 +\frac{C}{2}
|x-x_0|^2 \quad  \forall x \in  \overline B_{\delta}(x_0)\,.
\end{equation}
Furthermore, if $p\neq 0$ then the vector $\zeta$ can
be chosen so that $\pscal{\zeta}{p}\neq 0$.
\end{theorem}

\begin{remark}
By inspection of the proof below, one can derive some additional information on the constant $K$ and the vector $\zeta$
 appearing in \eqref{f:estidist0}. In fact,  let
 $p_0, \ldots, p_k $
be  points  in $\extr D^+ u(x_0)$ such that 
$p \in \conv\{p_0,\ldots,p_k\}$  ($1 \leq k \leq n$).  Then  the constant
 $K$  can be chosen as as the distance
between the origin and the boundary of the set
$\conv\{p_0-p,\ldots,p_k-p\}$, whereas 
the vector $\zeta$ 
can be chosen in the set
\begin{equation}\label{f:zeta}
Z:= \left\{\frac{z}{|z|}\, :\
z\in\conv\{p_0-p,\ldots,p_k-p\},\ z\neq 0
\right\}\,.
\end{equation}
\end{remark}

\bigskip
\textbf{Proof of Theorem \ref{t:estid}.}
Since $p \in D^+ u(x_0)\setminus \extr D^+ u(x_0)$
there exist $p_0, \ldots, p_k \in  \extr D^+ u(x_0)$
(with $1\leq k\leq n$)
and numbers 
$\lambda_0, \ldots, \lambda_k\in (0,1)$ with
$\sum _{i=0}^k \lambda_i = 1$ such that
$p = \sum_{i=0}^k \lambda_i p_i$.

We divide the remaining of the proof in three steps.

\medskip
\textsl{Step 1. 
The following inequality holds:}
\begin{equation}\label{f:estid12}
u(x) \leq u(x_0) + \min_{i=0,\ldots,k} \pscal{p_i}{x-x_0}
+\frac{C}{2} |x-x_0|^2\,,
\qquad \forall x\in\overline{B}_R(x_0).
\end{equation}

\smallskip
For every $i=0,\ldots,k$,
since $p_i\in D^+ u(x_0)$ 
by \cite[Prop.~3.3.1]{CaSi} we have that
\[
u(x) \leq u(x_0) + \pscal{p_i}{x-x_0} + \frac{C}{2}\, |x-x_0|^2
\qquad \forall x\in\overline{B}_R(x_0),
\]
so that \eqref{f:estid12} easily follows.

\medskip
\textsl{Step 2. 
Let $K$ denote the distance between the origin and the boundary
of $\conv\{p_0-p,\ldots,p_k-p\}$.
Then for every unit vector $\zeta$ in the set $Z$
defined in \eqref{f:zeta}, one has
\begin{equation}\label{f:estimin22}
\min_{i=0,\ldots,k} \pscal{p_i-p}{x}\leq
-K\, |\pscal{\zeta}{x}|\,,
\qquad \forall x\in\R^n.
\end{equation}
}

\smallskip
Since $\sum_{i} \lambda_i (p_i - p) = 0$
we have that the set
\[
F := {\rm span}\{p_0-p, p_1-p, \ldots, p_k-p\}
\]
is a subspace of $\R^n$ of dimension $k$.
Let
\[
Q := \conv\{p-p_0, p-p_1,\ldots, p-p_k\}\,;
\]
since $0$ belongs to the relative interior of the polytope $Q$,
and since $K$ is the distance between $0$ and
the boundary of $Q$, we clearly have $K > 0$
and $B := \overline{B}_K(0)\cap F \subseteq Q$.
Hence
\[
h_Q(x) := \max\{\pscal{q}{x}:\ q\in Q\}
\geq
\max\{\pscal{b}{x}:\ b\in B\} =: h_B(x),
\qquad \forall x\in\R^n.
\]
On the other hand, we have that
\[
h_Q(x) = \max_{i=0,\ldots k} \pscal{p-p_i}{x}
= - \min_{i=0,\ldots k} \pscal{p_i-p}{x}
\]
whereas, if $\zeta$ is any unit vector in the set $Z$ defined in \eqref{f:zeta},
then $\pm K\zeta\in B$, so that
\[
h_B(x) = \max\{\pscal{b}{x}:\ b\in B\}
\geq K |\pscal{\zeta}{x}|\,.
\]
Now \eqref{f:estimin22} 
easily follows.

\medskip
\textsl{Step 3. Completion of the proof.}

\smallskip
The estimate \eqref{f:estidist0} is a direct consequence
of \eqref{f:estid12} and \eqref{f:estimin22}.
In order to prove \eqref{f:estidist}
it is enough to observe that,
given $c > 0$,
the inequality
$K \, |t| \geq c\, t^2$
holds for every  $|t| < K/c$.
\qed

\bigskip
By exploiting the above geometric result for semiconcave functions, we obtain:

\begin{theorem}\label{t:diff}
Let $u\in C(\Omega)$ be a viscosity solution to
$-\Delta_{\infty}u = f(x,u)$ in $\Omega$.
If $u$ is locally semiconcave in $\Omega$,
then $u$ is everywhere differentiable (hence of class $C^1$)
in $\Omega$.
\end{theorem}

\begin{proof}
Assume by contradiction that $\Sigma(u)\neq \emptyset$.
Without loss of generality we can assume that $0\in \Sigma(u)$.
Let $p\in D^+u(0)\setminus \extr D^+u(0)$, $p\neq 0$.
By Theorem~\ref{t:estid}, there exists a unit vector 
$\zeta\in\R^n$ such that $\pscal{\zeta}{p}\neq 0$ and, for every $c>0$,
\[
u(x) \leq \varphi(x) :=
u(0) + \pscal{p}{x} - c \pscal{\zeta}{x}^2 +
\frac{C}{2}\, |x|^2,\qquad
\forall x\in B_{\delta}(0),
\]
with $\delta$ depending on $c$.
Since
\[
\Delta_{\infty} \varphi(0) =
-2 c  \pscal{\zeta}{p}^2 + C|p|^2,
\]
choosing $c>0$ large enough we get $-\Delta_{\infty} \varphi(0) > f(0, u(0))$,
a contradiction.

Since $u$ is differentiable everywhere in $\Omega$,
then by \cite[Prop.~3.3.4]{CaSi}
we conclude that $u\in C^1(\Omega)$.
\end{proof}

\bigskip
Finally, by combining Theorem \ref{t:34} and Corollary \ref{c:locsemiconc} with Theorem \ref{t:diff}, we obtain:

\begin{corollary}\label{c:contdiff}
Assume $\ipo$, and let $u$ be the solution to problem $(\ref{f:dirich})$. 
Then $u$ is continuously differentiable in $\Omega$.
\end{corollary}

\begin{proof}
By Theorem~\ref{t:34} we know that $u$ is locally semiconcave in $\Omega$,
hence  from Theorem~\ref{t:diff} we obtain that $u \in C ^ 1 (\Omega)$. 
\end{proof}

\section{The $P$-function along the gradient flow}\label{secP}

In this section we investigate the behavior of the $P$-function associated with problem 
(\ref{f:dirich}) according to the following 

\begin{definition}\label{d:P} Let $u$ be the solution to problem (\ref{f:dirich}). We set
\begin{equation}\label{defP}
P (x) := \frac{|\nabla u(x)| ^ 4}{4} + u(x) \, ,\qquad x \in{\Omega}\,.
\end{equation}
\end{definition}

\medskip
The relevance of this $P$-function in connection with problems \eqref{f:dirich} and \eqref{f:serrin} is enlightened by the next two lemmas.  
In fact, such relevance is two-fold. On one hand,  if it happens that  $P$ is constant on the whole $\Omega$, this gives geometric information on $\Omega$ (see Proposition \ref{p:P1} and Corollary \ref{corgeo}); we shall exploit this fact in Section \ref{secgeo} to study the overdedetemined boundary value problem \eqref{f:serrin}. On the other hand, if it happens that  $P$ is constant along a steepest ascent line of $u$, i.e. along a trajectory of the Cauchy problem 
\begin{equation}
\label{f:geo}
\begin{cases}
\dot \gamma (t) = \nabla u (\gamma (t)) 
\\
\gamma (0) = x \in \overline \Omega\,,& 
\end{cases}
\end{equation} then it is possible to compute explicitly  $u$ along the trajectory (see Proposition \ref{p:P2}); we shall exploit this fact in Section \ref{secreg} in order to obtain regularity thresholds for the solution to the Dirichlet problem \eqref{f:dirich}. 

We recall that the cut locus $\Cut (\Omega)$ and the high ridge $\high (\Omega)$ of $\Omega$ are the set defined as in (\ref{cut})-(\ref{high}); moreover, $\rho _\Omega$ and $\phi _\Omega$ denote respectively the inradius of $\Omega$ and the web function introduced in (\ref{defphi}).

\begin{proposition}
\label{p:P1}
Assume that the unique solution $u$ to the Dirichlet problem~\eqref{f:dirich}
is of class $C^1(\Omega)$, and
that the $P$-function introduced in Definition~\ref{d:P} satisfies
\begin{equation}
\label{f:pconst}
P (x) = \lambda \qquad \hbox{ for $\mathcal L ^n$-a.e. } x \in \Omega\,,
\end{equation}
for some $\lambda\leq c_0 \rho_\Omega^{4/3}$.
Then $\lambda = c_0 \rho_\Omega^{4/3}$ and
$u=\phi _\Omega$, where $\phi_\Omega$ is the function
defined in~\eqref{defphi}, 
and it holds $\Cut (\Omega) = \high (\Omega)$. 
\end{proposition} 

\begin{proof}
It is clear that $\lambda = \max u$.
On the other hand, $\max u \geq \max v = c_0 \rho_\Omega^{4/3}$,
where $v$ is the radial solution of the Dirichlet problem
in a ball $B_{\rho_\Omega}\subseteq\Omega$.
Hence $\lambda = c_0 \rho_\Omega^{4/3}$.

Let $H\colon\R\times\R^n\to\R$ be the Hamiltonian defined by
\[
H(u, p) := \frac{1}{4} |p|^4 + u - \lambda.
\]
Then the equality~\eqref{f:pconst} can be rewritten as
\begin{equation}
\label{star}
H(u(x), \nabla u(x)) = 0,  \qquad \mathcal L ^n \hbox{-a.e. on } \Omega\ .
\end{equation}
Since $u$ is of class $C^1(\Omega)$, then it follows that it is a classical
(hence also a viscosity) solution of the Dirichlet problem
\begin{equation}
\label{f:dpH}
\begin{cases}
H(u, \nabla u) = 0, & \text{in}\ \Omega,\\
u = 0, & \text{on}\ \partial\Omega.
\end{cases}
\end{equation}
Since the solution to this Dirichlet problem is unique
(see {\it e.g.}\ \cite[Theorem III.1]{B}),
to prove that $u=\phi_\Omega$ it is enough to show that also
$\phi_\Omega$ is a viscosity solution to~\eqref{f:dpH}.

It is readily seen that $\phi_\Omega$ is differentiable at every point
$x\in\Omega\setminus S$, where $S := \Sigma(\Omega) \setminus M(\Omega)$,
and $H(\phi_\Omega(x), \nabla\phi_\Omega(x)) =0$.

Hence it is enough to show that, for every $x\in S$,
one has
\begin{equation}
\label{f:ssd}
H(\phi_\Omega(x), p) \leq 0, \quad \forall p\in D^+ \phi_\Omega(x), \qquad
H(\phi_\Omega(x), q) \geq 0, \quad \forall q\in D^- \phi_\Omega(x),
\end{equation}
where the symbols $D^+ \psi$ and $D^-\psi$ denote respectively the super and sub-differential
of a function $\psi$.
Since $x\in S$, then $x\not\in M(\Omega)$, so that $\delta := \dist_{\partial\Omega} (x) < \rho_\Omega$.
As a consequence
\[
D^{\pm} \phi_\Omega(x) = g'(\delta)\, D^{\pm} \dist_{\partial\Omega}(x),
\]
where $g(t) := c_0 [\rho_\Omega^{4/3} - (\rho_\Omega - t)^{4/3}]$ and
$g'(\delta) > 0$.
In particular $D^- \phi_\Omega(x) = \emptyset$, since the sub-differential of the
distance function is empty at singular points
(see \cite[Corollary~3.4.5]{CaSi}),
so that the second condition in~\eqref{f:ssd} is trivially satisfied.

Let now consider $p\in D^+\phi_\Omega(x)$.
Since $D^+\dist_{\partial\Omega}(x)$ is contained in the closed unit ball,
there exists $\xi\in\R^n$, $|\xi|\leq 1$, such that
$p = g'(\delta)\, \xi$, hence
\begin{equation}
\label{f:last}
H(\phi_\Omega(x), p) = \frac{1}{4}\, g'(\delta)^4\, |\xi|^4 + \phi_\Omega(x) - \lambda
\leq \frac{1}{4}\, g'(\delta)^4 + g(\delta) - \lambda.
\end{equation}
On the other hand, if $(x_k)\subset \Omega\setminus\Sigma(\Omega)$ is a sequence
converging to $x$, we get
\[
0 = H(\phi_\Omega(x_k), \nabla\phi_\Omega(x_k))
= \frac{1}{4} g'(\dist_{\partial\Omega}(x_k))^4 + g(\dist_{\partial\Omega}(x_k)) - \lambda
\longrightarrow \frac{1}{4}\, g'(\delta)^4 + g(\delta) - \lambda,
\]
that, together with~\eqref{f:last},
proves the first condition in~\eqref{f:ssd}.
\end{proof}

\begin{proposition}\label{p:P2} Let $u$ be the solution to problem \eqref{f:dirich}, and let
 $\gamma :[0, \delta ) \to \Omega$ be a local solution to problem \eqref{f:geo}, starting at a point $x$ with $\nabla u (x) \neq 0$.  
 Assume that $u$ is differentiable at $\gamma (t)$ for $\mathcal L ^1$-a.e.\ $t \in [0, \delta)$, and 
that the $P$-function introduced in Definition \ref{d:P} satisfies
$$P (\gamma (t)) = \lambda \qquad \hbox{ for $\mathcal L ^1$-a.e.\ $t \in [0, \delta)$ }\,.$$
Then, setting $m := u (x)$, it holds $\lambda >m$ and, for all $t \in [0, \delta)$,  the function $\varphi (t):=u (\gamma (t))$ agrees with the function
\[
\overline{\varphi}(t) :=
\begin{cases}
\lambda - (\sqrt{\lambda-m} - t)^2,
&\text{if}\ t\in [0, \sqrt{\lambda-m})\\
\lambda,
& \text{if}\ t\geq \sqrt{\lambda-m}\,.
\end{cases}
\] 
\end{proposition} 

\proof  The assumed equality $P (\gamma (t)) = \lambda$ $\mathcal L ^1$-a.e.\ on $[0, \delta)$ implies that $\lambda>m$ (since $\gamma$ is a steepest ascent line of $u$ and $\nabla u (x) \neq 0)$.
Moreover, the function
$\varphi (t) := u (\gamma (t))$ 
is in $AC ([0, \delta))$, because
$u \in {\rm Lip} (\Omega)$ (see e.g.\cite[Lemma 2.9]{ArCrJu}) and $\gamma\in AC ([0, \delta))$,
and it is a solution to the
Cauchy problem
\[
\begin{cases}
\dot{\varphi}(t) = 2\sqrt{\lambda - \varphi(t)} \qquad \mathcal L ^1\text{-a.e. on } [0, \delta)\\
\varphi(0) = m.
\end{cases}
\]
It is readily seen that this Cauchy problem admits
a unique global solution, given precisely by $\overline \varphi$. 
Therefore, we conclude that $\varphi$ agrees with $\overline \varphi$ on $[0, \delta)$. 
\qed

\bigskip
We start now investigating what can be said about the behaviour $P$  along trajectories of problem \eqref{f:geo} and globally over $\Omega$. 
A first elementary observation comes from the pde interpreted pointwise at points of two-differentiability of $u$:

\begin{lemma}\label{l:easy} Let $u$ be the solution to problem \eqref{f:dirich}, and let
$\gamma :[0, \delta ) \to \Omega$ be a local solution to problem \eqref{f:geo}. Assume that $u$ is twice differentiable at $\gamma (t)$ for $\mathcal L ^1$-a.e.\ $t \in [0, \delta)$.  Then it holds
\begin{equation}\label{f:dernulla}
\frac{d}{dt}\big( P (\gamma (t)) \big )  = 0 \qquad \mathcal L ^ 1\hbox{-a.e.\ in } [0, \delta)\,.
\end{equation}
\end{lemma} 

\proof
At every point $x$ where $u $ is twice differentiable, it holds
\[
\nabla P(x)=|\nabla u (x)|^2 D^2u(x)\, \nabla u(x)+\nabla  u(x)\, ;$$
we infer that
$$\begin{array} {ll} \langle \nabla P  (x),\nabla u (x) \rangle&=|\nabla u(x)|^2  \langle D^2 u (x)   \nabla u(x) ,\nabla u(x) \rangle+|\nabla u(x)  |^2 \\ \noalign{\medskip} & =
|\nabla u(x) |^2\left(\Delta_\infty u(x) +1\right)= 0\, .
\end{array}
\]

Thus, since by assumption $u$ is twice differentiable at $\gamma (t)$ for $\mathcal L ^1$-a.e.\ $t \in [0, \delta)$, it holds
$$
\frac{d}{dt}\big( P (\gamma (t)) \big ) 
=\Big \langle \nabla P  (\gamma (t) ),
\frac{d}{dt} \gamma  (t) \Big\rangle 
= \Big \langle \nabla P (\gamma (t)), \nabla u
(\gamma (t) ) \Big \rangle =  0 \,
$$
for $\mathcal L ^1$-a.e.\ $t \in [0, \delta)$.
 \qed

\bigskip The main and crucial difficulty in exploiting the information given by Lemma \ref{l:easy}
is the possible lackness of regularity of the function $P \circ \gamma$. In particular, we are not able to ensure that this map is in $AC([0,\delta))$, so to infer from \eqref{f:dernulla} that $P$ is constant along $\gamma$. 

Moreover, in order to obtain some information on the {\it global} behavior of $P$ on $\Omega$, and not merely on a single trajectory, we need to
study the solutions to problem \eqref{f:geo}
under the  aspects of local uniqueness and 
continuation property.

This is the main reason why hereafter we require both 
the assumptions $\ipo$-$\ipu$. Several comments on these assumptions are postponed after the main result of this section, which reads:

\begin{theorem}\label{t:ineqP}
Assume $\ipo$--$\ipu$. Then  the P-function introduced in Definition \ref{d:P} satisfies
\begin{equation}\label{f:tesiP}
\min _{\partial \Omega} \frac{|\nabla u| ^ 4 }{4} \leq P (x)  \leq \max _{\overline {\Omega}} u
\qquad \forall  x \in \overline \Omega \,.
\end{equation}
\end{theorem}

\begin{remark} \label{r:H12} 

(i) It is a natural question to ask whether  $\ipo$ implies $\ipu$. Indeed, 
under the assumption $\ipo$, the set $\Omega$ is of
class $C^1$ and, by Corollary~\ref{c:contdiff}, the unique viscosity solution $u$
to problem \eqref{f:dirich} belongs to $C ^ 1 (\Omega)$. Hence,  if we assume $\ipo$, requiring the additional
condition $\ipu$ amounts to ask just that $u$ is of class $C ^1$ {\it up to the boundary}. 
In view of the results proved in \cite{WaYu, Hong, Hong2}, it seems reasonable to guess that the latter condition 
is not always fulfilled when $\ipo$ holds, but becomes true under the  additional regularity requirement that $\partial \Omega$ is $C^2$.
However, this boundary regularity property seems to be a highly technical point, 
which goes beyond the purposes of this paper, and we limit ourselves to address it as an interesting open problem. 

(ii)  Under the assumptions $\ipo$--$\ipu$, if $u$ is the unique viscosity solution 
to problem \eqref{f:dirich}, the set of critical points of $u$ agrees  with the set $\argmax_{\overline \Omega} u $ where $u$ attains its maximum over $\overline \Omega$. Indeed, by Theorem~\ref{t:34}, the
function $u^{3/4}$ is concave (and strictly positive) in $\Omega$; hence
its gradient 
vanishes only on points of
maximum of $u$. 

(iii)  Since 
the unique solution
$u$ to \eqref{f:dirich} is strictly positive in $\Omega$ (cf.\ Remark \ref{r:pos}),
it can be extended to a continuously differentiable function,
still denoted by $u$,
in an open set $\widetilde \Omega \supset \overline{\Omega}$,
such that $u < 0$ in $\widetilde \Omega \setminus\overline{\Omega}$.
Consequently, 
also the $P$-function in \eqref{defP}  can be extended to a continuous function in $\widetilde \Omega$, still denoted by $P$. 
\end{remark}

The remaining of this section is devoted to the proof of Theorem \ref{t:ineqP}.

In the sequel, we set for brevity
\begin{equation}\label{f:maxu}
K := \argmax_{\overline{\Omega}} u\, , \qquad \mu:= \max _{\overline {\Omega}} u\,. 
\end{equation}

A first key step is the construction of the \textsl{gradient flow} $\X$ associated with $u$, and the location  
of its terminal points:

\begin{lemma}\label{l:geo} 
Assume $\ipo$--$\ipu$, and let  $x\in \overline \Omega \setminus K$.  Then
there exists a unique solution $\X(\cdot, x)$
to~\eqref{f:geo}
defined in a interval $[0, T(x))$,
where $T(x)\in (0, +\infty]$ is given by
\begin{equation}
\label{f:defT}
T(x) := \sup\{t\geq 0:\ \nabla u (\X(t,x))\neq 0 \}\,.
\end{equation}
Moreover, 
\begin{equation}
\label{f:limu}
\lim_{t\to T(x)^-} \X(t, x)
\in K,\qquad
\lim_{t\to T(x)^-} \nabla u (\X(t,x)) = 0\,.
\end{equation}
Finally, there exist $x_0\in\partial\Omega$
and $t_0\in [0, T(x_0))$ such that
$x = \X(t_0, x_0)$.
\end{lemma}
  
\proof 
The local uniqueness 
of forward solutions to~\eqref{f:geo} 
follows from 
Corollary~\ref{c:locsemiconc} and
\cite[Theorem 3.2 and Example 3.6]{CaYu}. 
(Note that 
we think of $u$ as extended to $\widetilde \Omega \supset \overline \Omega$ according to Remark \ref{r:H12} (iii).)
Moreover,
\begin{equation}
\label{f:monot}
\frac{d}{dt} u(\gamma(t))
= \nabla u(\gamma(t))\cdot\dot{\gamma}(t)
= |\nabla u(\gamma(t))|^2 \geq 0. 
\end{equation}
Hence  local forward solutions to~\eqref{f:geo}  remain into $\overline{\Omega}$ for every $t$ at which they are defined. 
As a consequence, 
they are actually global forward solutions, i.e., they are defined in  $[0,+\infty)$ (and remain into $\overline \Omega$ for all $t \in [0, + \infty)$). 

For a fixed $x \in \overline \Omega \setminus K$, by local uniqueness, all forward solutions to \eqref{f:geo}
must coincide in the interval $[0, T(x))$, 
with $T (x)$ defined as in \eqref{f:defT}: from now on, 
we denote by  $\X(\cdot, x)$ \textsl {the unique solution} to \eqref{f:geo} in $[0, T(x))$. 

Now, in order to prove \eqref{f:limu}, we distinguish the two cases $T (x) < + \infty$ and $T (x) = + \infty$. 
 
If $T(x) < +\infty$, we have
$$\lim _{t \to T (x) ^ -} \X(t, x)\in K\, , $$
so that 
$$\lim _{t \to T (x) ^ -} \nabla u(\X(t, x)) =0\,.$$
If $T (x) = + \infty$, assume by contradiction that
\begin{equation}\label{f:max}
m:=  \lim _{t \to +\infty}u ( \X(t, x) ) < \mu
\end{equation}
(note that the limit which defines $m$ exists by monotonicity in view of 
\eqref{f:monot}). 
Since $\nabla u$ is continuous and strictly positive in the compact set $ \{x \in \overline \Omega \ :\ u (x) \leq m \}$, 
from \eqref{f:monot} there exists $\alpha >0$ such that 
\[
\frac{d}{dt} u(\X(t, x)) \geq \alpha > 0
\qquad\forall t\in [0,  + \infty)\,,  
\]
which clearly contradicts  \eqref{f:max}. 
Therefore, it must be $m= \mu$, and,
taking into account that
$\frac{d}{dt} \X(t, x) $ is bounded, 
\eqref{f:limu} is proved also in case $T (x) = + \infty$. 

Now let $\gamma\colon (T^-, +\infty)$,
$T^- \in [-\infty, 0)$,
be a maximal
solution to \eqref{f:geo} in
$\Omega$.
(From the discussion above we have 
$\gamma(t) = \X(t,x)$ for every $t\in [0, T(x))$.)
By \eqref{f:monot}, the map
$t\mapsto u(\gamma(t))$ is strictly monotone increasing
for $t\in (T^-, T(x))$, so that
\[
\gamma(t) \in K' := \{y\in\overline{\Omega}:\
u(y)\leq u(x)   \}\,,
\qquad \forall t\in (T^-, 0].
\]
Since $\nabla u$ is continuous and $\nabla u \neq 0$
in the compact set $K'$, we have that
\[
\frac{d}{dt} u(\gamma(t)) \geq \alpha' > 0
\qquad\forall t\in (T^-, 0].
\]
Hence $T^-$ must be finite and,
being $\dot{\gamma}$ bounded, 
$\lim_{t\to T^-} \gamma(t) =: x_0 \in \partial\Omega$.
By construction, it holds
\[
\X(t, x_0) = \gamma(T^- + t)\qquad
\forall t\in (0, T(x_0)),
\]
hence setting $t_0 := - T^- > 0$ we have that
$\X(t_0, x_0) = \gamma(0) = x$.
\qed

\bigskip 
As we have already mentioned after Lemma \ref{l:easy}, 
such result cannot be directly exploited to infer the constancy of the $P$-function along 
the flow $\X$, because of the possible lack of absolute continuity of $P$. 
In order to overcome this difficulty, we approximate $u$ via its supremal convolutions, defined for $\e >0$ by
\begin{equation}\label{f:ue}
u ^ \e (x) := \sup _{y \in \R ^n} \Big \{ \tilde u (y)  - \frac { |x-y| ^ 2 }{2 \e}  \Big \} \qquad \forall x \in \R ^n\,,
\end{equation}
where $\tilde u$ is a Lipschitz extension of $u$ to $\R ^n$ 
with ${\rm Lip}_{\R ^n} (\tilde u) = {\rm Lip} _{\overline \Omega} (u)$. 

In the next lemma we state the basic properties of the functions $u ^ \e$ which we are going to use in the sequel. 

Let us recall that, 
according to \cite[Lemma 3.5.7]{CaSi}, there exists $R>0$, depending only on ${\rm Lip}_{\R ^n} (\tilde u)$, such that any point $y$ at which the supremum in (\ref{f:ue}) is attained satisfies $|y-x| < \e R$. 
Thus, setting
\begin{equation}\label{f:Ue} 
U _\e:= \big \{ x \in \Omega \ :\ u (x) > \e \big \} \, , \qquad A_\e := \big \{ x \in U _\e \ :\ d_{\partial U _\e}(x) > \e R \big \}\, ,
\end{equation}
there holds
\begin{equation}\label{d:ue2}
u ^ \e (x) = \sup _{y \in U _\e} \Big \{ u (y)  - \frac { |x-y| ^ 2 }{2 \e}  \Big \} \qquad \forall x \in A_\e\,. 
\end{equation}
Moreover, let us define
\begin{equation}\label{f:omegae}
m_\e := \max_{\partial A_\e} u^\e,
\qquad
\Omega_\e := \{x\in A_\e : \ u^\e (x) > m_\e \}\,.
\end{equation}

\begin{lemma}\label{l:approx1} 
Assume $\ipo$--$\ipu$.
Let $u ^\e$ and $\Omega _\e$ be defined respectively as in \eqref{f:ue} and \eqref{f:omegae}. 
Then:
\begin{itemize}
\item[(i)] $u ^ \e$ is of class $C ^ {1, 1}$ on $\Omega _\e$;
\item[(ii)] $u ^\e$ is a sub-solution to $-\Delta _\infty u -1 = 0$ in $\Omega_\e$;
\item[(iii)] as $\e \to 0 ^+$,  it holds
\[
\begin{array}{ll}
&  u ^ \e \to u \qquad \hbox{ uniformly in } \overline \Omega, \\ \noalign{\smallskip}
& \nabla u ^ \e \to \nabla u   \qquad \hbox{ uniformly in } \overline \Omega 
\end{array}
\]
(so that $m_\e \to 0$ and $\Omega_\e$ converges to $\Omega$
in Hausdorff distance). 
\end{itemize}
\end{lemma}

\proof
In order to prove (i), by \cite[Corollary 3.3.8]{CaSi}, it is enough to show that  $u ^\e$ is both semiconcave   and semiconvex on $\Omega _\e$. We have $u ^ \e = - ( -u ) _\e$, where $( -u ) _\e$ is the infimal convolution defined by
\[ 
( -u ) _\e (x) := \inf _{y \in U _\e} \Big \{-  u (y)  + \frac { |x-y| ^ 2 }{2 \e}  \Big \} \qquad \forall x \in \Omega _\e\,.
\]
{}From \cite[Proposition 2.1.5]{CaSi}, it readily follows that $(-u) _\e$ is semiconcave on $\Omega _\e$, and hence that 
$u ^ \e$ is semiconvex on $\Omega _\e$. In order to show that $(-u) _\e$ is semiconvex on $\Omega _\e$ (and hence that $u ^ \e$ is semiconcave on $\Omega _\e$), let $x_i \in \Omega _\e$, $i=1, 2$, be fixed, and let $y _i \in U _\e$ be points where the infima which define $( -u )_\e(x_i)$ are attained. 
Denoting by $C_\e$ the semiconcavity constant of $u$ on $U _\e$, we have
\[
\begin{array}{ll}
& (-u) _\e (x_1) + (-u) _\e (x_2 ) - 2 ( -u ) _\e \Big (\frac{x_1 + x_2} {2} \Big )  \geq \\ 
\noalign{\medskip}
& - u (y _1)  - u (y _2) + 2 u \Big (\frac{y_1 + y_2} {2} \Big ) + \frac{|x_1 - y _1| ^ 2 }{2 \e} 
+ \frac{|x_2 - y _2| ^ 2 }{2 \e}   - \frac{2}{2 \e} \Big | \frac{x_1 + x_2}{2} -  \frac{y_1 + y_2}{2}  \Big | ^ 2 \\ 
\noalign{\medskip}
& \geq - C _\e \Big | \frac{y _1 - y _2}{2} \Big | ^ 2 + \frac{1}{2 \e} \Big ( |x_1 - y _1| ^2 + |x_2 - y _2| ^ 2 - 2 
 \Big | \frac{x_1 + x_2}{2} -  \frac{y_1 + y_2}{2}  \Big |^ 2\Big )  \\
 \noalign{\medskip}
& \geq - \frac{  2 C_\e}{ 2 - \e C_\e} |x_1 - x _2| ^ 2 \,.
\end{array}
\]
Thus $(-u)_\e$ is semiconvex with constant  $\frac{  2 C_\e}{ 2 - \e C_\e}$ on $\Omega _\e$. 

\medskip
\noindent
Let us now prove (ii). Let $x \in \Omega _\e$, and let $(p, X) \in  J  ^ {2, +} _{\Omega _\e} u ^ \e (x)$. It follows from magical properties of supremal convolution ({\it cf.} \cite[Lemma A.5]{CHL}) that $(p, X) \in J ^ {2, +} _{\Omega _\e} u (y)$, where $y$ is a point at which the supremum which defines
$u ^ \e (x)$ is attained.  Since $y \in U _\e \subset \Omega _\e$, it holds $J ^ {2, +} _{\Omega} u (y)= J  ^ {2, +} _{\Omega _\e} u ^ \e(x)$; therefore, we have  $(p, X) \in J ^ {2, +} _{\Omega} u (y)$, which implies $- \langle Xp, p \rangle- 1 \leq 0$. 

Finally, let us turn to the proof of the convergence properties (iii). 
For the uniform convergence of $u ^\e$ to $u$ in $\overline \Omega$ , see for instance \cite[Thm.\ 3.5.8]{CaSi}. 
In order to show the uniform convergence of the gradients, recall first that there holds
$$\nabla u ^ \e (x) = \nabla u ( y_\e (x)) \qquad \forall  x \in \Omega _\e \, ,$$
where $y _\e (x)$ is a point where the supremum which defines $u ^ \e (x)$ is attained \cite[Thm. 3.1 (a)]{GZ}.  Moreover, as already mentioned above, there exists $R>0$, depending only on ${\rm Lip} _{\overline \Omega} (u)$, such that 
$|y _\e (x) - x | < \e R$ for every $x \in \overline \Omega$; in particular, for $x \in \Omega _\e$, there holds $y _\e (x) \in U _\e \subset \Omega$. Then we can write:
\[
\begin{split}
\lim _{\e \to 0 ^ +} \sup _{x \in \overline \Omega} |\nabla u ^ \e (x) - \nabla u (x)| 
 & = 
\lim _{\e \to 0 ^ +} \sup _{x \in \Omega_\e} |\nabla u  ^\e (x ) - \nabla u (x)| 
\\ & = 
\lim _{\e \to 0 ^ +} \sup _{x \in \Omega_\e} |\nabla u  ( y _\e (x) ) - \nabla u (x)| 
\\ & 
\leq {\rm Lip} _{\overline \Omega} (u) \lim _{\e \to 0 ^ +} | y _\e (x ) - x| 
\\ &
\leq {\rm Lip} _{\overline \Omega} (u)  
\lim_{\e\to 0^+} (\e R) = 0 \,.
\hfill\qedhere
\end{split}
\]

\bigskip

Next we observe that, for every $\e >0$, one can consider the gradient flow $\Xe$ 
associated with $u^\e$. Namely, 
for every $x_\e\in\overline{\Omega_\e}$,
the Cauchy problem
\begin{equation}\label{f:geoe}
\begin{cases}
\dot \gamma_\e (t) = \nabla u ^ \e (\gamma _\e (t))\,,
\\
\gamma_\e  (0) = x _\e \in  \overline{\Omega_{\e}}\,, 
\end{cases}
\end{equation}
admits a unique solution 
$\Xe(\cdot, x_\e)\colon [0, +\infty) \to \overline{\Omega_\e}$. Indeed, 
the fact that $\Xe(\cdot, x_\e)$ is
defined in $[0, +\infty)$ follows from the estimate
\[
\frac{d}{dt}u^\e (\gamma_\e(t)) 
= |\nabla u^\e(\gamma_\e(t))|^2 \geq 0,
\] 
so that $\gamma_\e (t)\in\overline{\Omega_\e}$ for
every $t\geq 0$,
while uniqueness follows from the
$C^{1,1}$ regularity of $u^\e$
stated in Lemma~\ref{l:approx1}(i).

In the following key lemma, we establish the behavior, along the  flow $\X _\e$, 
of the approximate $P$-function defined by
 \begin{equation}\label{f:Pe}
P_\e (x) := \frac{|\nabla u^\e(x) | ^ 4}{4} + u ^\e(x) \, ,\qquad x \in \overline{\Omega_\e}\,.
\end{equation}
In fact we show that $P _\e$ increases along $\X _\e$:

\begin{lemma}\label{l:approx2} 
Assume $\ipo$--$\ipu$.
Let $u ^\e$, $\Omega _\e$,  and $P _\e$ be defined respectively as in \eqref{f:ue}, \eqref{f:omegae},  and \eqref{f:Pe}. 
Then, 
for ${\mathcal H} ^ {n-1}$-a.e.\ $x_\e \in\partial{ \Omega _\e}$,
it holds
\[
P _\e(\Xe(t_1, x_\e)) \leq P _\e (\Xe (t_2, x_\e) ) \,  \qquad \forall  \, t_1, t_2  \, \hbox{ with }  0\leq t_1 \leq t_2
\,.
\]
\end{lemma}

\proof

\medskip
At every point $x$ where $u ^\e$ is twice differentiable, it holds
$$
\nabla P_\e(x)=|\nabla u^\e(x)|^2 D^2u^\e(x)\, \nabla u^\e(x)+\nabla ^\e u(x)\, ;$$
we infer that
$$\langle \nabla P _\e (x),\nabla u^\e (x) \rangle= 
|\nabla u^\e(x) |^2\left(\Delta_\infty u^\e(x) +1\right)\geq 0\, ,
$$
where the last inequality follows from  Lemma \ref{l:approx1}(ii).  

Thus, if $u ^ \e$ is twice differentiable at 
$\Xe (t, x_\e)$, we have
\begin{equation}\label{f:derneg}
\frac{d}{dt}\big( P_\e (\Xe (t, x_\e)) \big ) 
 = \Big \langle \nabla P _\e(\Xe (t, x_\e)), \nabla u^\e
(\Xe (t, x_\e)) \Big \rangle \geq 0 \, .
\end{equation}

Let us now show that, for  ${\mathcal H} ^ {n-1}$-a.e.\ $x_\e \in \partial \Omega _\e$, the inequality (\ref{f:derneg}) is satisfied  ${\mathcal L} ^ {1}$-a.e.\ on 
$[0, +\infty)$.
To that aim we have to prove that,  for  ${\mathcal H} ^ {n-1}$-a.e.\ $x_\e \in \partial \Omega _\e$,  $u ^\e$ is
twice differentiable at $\Xe (t, x_\e)$ for ${\mathcal L} ^ {1}$-a.e.\  $t \in [0, +\infty)$. 
Namely we have to show that, 
setting 
\[
N (x_\e) :=  \Big \{ t \geq 0 \ :\ u ^\e  \text{ is not twice differentiable at }  \Xe (t, x_\e) \Big \}\, ,\qquad  x_\e \in \partial \Omega _\e\, , 
\]
it holds
\begin{equation}\label{f:null}
{\mathcal L} ^ {1} (  N (x_\e) ) = 0 \qquad \hbox{ for } {\mathcal H} ^ {n-1} \text{-a.e.\ } x_\e \in \partial \Omega _\e\, .
\end{equation}

By construction the set
\[
E_\e:= \Big \{ \Xe (t, x_\e) \ :\ x_\e \in \partial \Omega _\e\, ,\ t \in N (x_\e) \Big \}
\]
is contained into the set of points where $u^\e$ is not twice differentiable.

By Lemma \ref{l:approx1}(i), we know that $u ^ \e$ is twice differentiable ${\mathcal L}^n$-a.e.\ on $\Omega _\e$. 
This implies that the set $E_\e$ is Lebesgue negligible. By the area formula, we have
\[
0 = {\mathcal L } ^n (E_\e) = \int _{\partial \Omega _\e} \, d {\mathcal H} ^ {n-1} (x_\e) 
\int _{N (x_\e)} J \Xe(t, x_\e) \, dt  \,,
\]
where $J \Xe$ is the Jacobian of the function $\Xe$
with respect to the second variable.
Since this Jacobian is strictly positive ({\it cf.}\ \cite[eq.\ (5)]{AmCr}), we infer that (\ref{f:null}) holds true. 
Hence,  for  ${\mathcal H} ^ {n-1}$-a.e.\ $x_\e \in \partial \Omega _\e$, the inequality (\ref{f:derneg}) is satisfied  ${\mathcal L} ^ {1}$-a.e.\ on $[0, +\infty)$.
Then, for   ${\mathcal H} ^ {n-1}$-a.e.\ $x_\e \in \partial \Omega _\e$,
integrating \eqref{f:derneg} over $[t_1, t_2]$ and taking into account that, by Lemma \ref{l:approx1}(i), the map 
$t \mapsto P _\e \circ \Xe(t, x_\e)$ is 
locally Lipschitz continuous on 
$[0, +\infty)$, the lemma is proved.  
 \qed

\bigskip
\begin{remark}
We point out that an analogous procedure as the one adopted above does not work with the infimal convolutions $u _\e$ in place of the supremal convolutions $u ^ \e$. The reason
is simply the fact that such infimal convolutions are not necessarily of class $C ^ {1,1}$. 
\end{remark}

We are finally in a position to give the

\medskip
{\bf Proof of Theorem \ref{t:ineqP}}.  
Let us show firstly that 
\begin{equation}\label{f:prelim}
\min _{\partial \Omega} \frac{|\nabla u| ^ 4 }{4} \leq \mu \,.
\end{equation}
Let $B= B _\rho (z_0)$ be a ball internally tangent to $\Omega$ at $y _0$, and let $\phi _B$ denote the unique solution to (\ref{f:dirich}) on $B$. 
By the comparison principle proved in \cite[Thm. 3]{LuWang}, it holds $u \geq \phi _B$ on $B$. Hence, 
$$u (z_0) \geq \phi _B (z_0) = c_0 \rho ^ {4/3}$$
and
$$|\nabla u (y _0) |\leq |\nabla \phi _B (y _0) |  = ( 3 \rho ) ^ {1/3}\,.$$
Therefore, we have
$$\min _{\partial \Omega} \frac{|\nabla u| ^ 4 }{4} \leq 
 \frac{|\nabla u|^ 4 (y_0) }{4}   \leq c_0 \rho ^ {4/3} \leq u (z_0) \leq 
\mu \,.$$

Let us now prove  \eqref{f:tesiP}.
It is enough to show that  \eqref{f:tesiP} holds for all 
$x \in \Omega \setminus K$ (otherwise $P(x) = u (x)$ and we are done by the inequality \eqref{f:prelim}). 
By Lemma~\ref{l:geo},
given $x\in\Omega\setminus K$, there exist
$
x_0\in\partial\Omega $ and 
$t_0\in[0, T(x_0))$ such that $x = \X (t_0, x_0)$. 
By Lemma~\ref{l:approx2}, we may find a sequence of points
 $x_\e\in \partial\Omega_\e$ converging
to $x_0$ such that, 
for every $t\geq t _0$, we have
\[
P_\e(x_\e) \leq
P_\e (\Xe(t_0, x_\e)) \leq P_\e(\Xe(t, x_\e))\,.
\] 
We now pass to the limit as $\e \to 0 ^+$ in the above inequalities: 
by using the continuous dependence for ordinary differential equations
(see e.g.\ \cite[Lemma~3.1]{Hale}), 
and the uniform convergences stated in
Lemma~\ref{l:approx1}(iii), we get
\begin{equation}\label{f:pineq}
P(x_0) \leq P(x) \leq P(\X(t, x_0))\,.
\end{equation}
We have
\[
P(x_0) = \frac{|\nabla u(x_0)|^4}{4}
\geq \min_{\partial\Omega} \frac{|\nabla u|^4}{4};
\]
on the other hand, from \eqref{f:limu}, it holds
\[
\lim_{t\to T(x_0)^-}
P(\X(t, x_0)) = 
\lim_{t\to T(x_0)^-}
u(\X(t, x_0))
\leq \mu\,.
\]
Then \eqref{f:tesiP} follows from \eqref{f:pineq}.
\qed 

\section{Geometric results for Serrin's problem} \label{secgeo}

\begin{theorem}\label{t:serrin2} Assume $\ipo$--$\ipu$. 
Further, assume that there exists an inner  ball $B$ of radius $\rho _\Omega$ which meets $\partial \Omega$ at two points
lying on the same diameter of $B$.  
If the overdetermined boundary value problem $(\ref{f:serrin})$ admits a solution $u$, then it holds $u= \phi _\Omega$ (with $\rho _\Omega = a ^ 3/3$), and 
$\Cut (\Omega ) = \high (\Omega)$. 
\end{theorem}

\proof   Let $B= B _{\rho _\Omega}(x_0)$ be an inner ball of radius $\rho _\Omega$ which meets $\partial \Omega$ at two diametral points $y _\pm$. 
We claim that 
there exists a domain $D$ with $\Cut (D) = \high (D)$  (and $\rho _D = \rho _\Omega$), 
such that
\begin{equation}\label{f:condD}
B \subset \Omega \subset D \,, \qquad \partial B \cap \partial \Omega = \partial B \cap \partial D  = \{ y_+, y_- \} \,. 
\end{equation}

Namely,  assume without loss of generality that the center $x_0$ of $B$ is  the origin, and that $y _\pm = \pm \rho _\Omega e _n$.  
Let $S \subset \{ x_n = 0 \}$ be a $(n-1)$-dimensional disk centered at the origin, with radius sufficiently large so that
$\Omega$ is contained into the cylinder $S \times [ - \rho _\Omega, \rho _\Omega]$. 
Then the conditions (\ref{f:condD}) are satisfied by taking 
$$D:= \{ x \in \R ^n\ :\ d_{S } (x)  < \rho _\Omega \}\,.$$

Now, denote by $\phi _B$ and $\phi _D$ the web functions  defined according to (\ref{defphi}) (respectively on $B$ and on $D$), 
and by  $\gamma$ the diameter of $B$ containing $y_+$ and $y_-$. 
Since $-\Delta _\infty u = -\Delta _\infty \phi _B= 1$ on $B$, and $u \geq 0= \phi _B$ on $ \partial B$, by the comparison principle proved in \cite[Thm. 3]{LuWang}, it holds $u \geq \phi _B$ on $B$. 
In the same way, we get the inequality $u \leq \phi _D$ on $D$. 
We thus have
\begin{equation}\label{f:ineq}\phi _B (x) \leq u (x) \leq \phi _D (x) \qquad \forall x \in B \,.\end{equation}

We can deduce several consequences from these inequalities. 
Firstly we observe that, since both the functions $\phi _B$ and $\phi _D$ have a relative maximum at $x_0$, by (\ref{f:ineq}) the same property holds true for $u$. Hence $x_0$ is a critical point of $u$.  In turn, by Remark \ref{r:H12}, we know that 
\begin{equation}\label{f:argmax}
x_0 \in K\,. 
\end{equation}

Moreover we notice that, since the distance functions $d_{\partial B}$ and $d _{\partial D}$ agree on $\gamma$ (as both coincide with $d _{\partial \Omega}$), there holds
\begin{equation}\label{f:equ} \phi _B (x) = \phi _D (x) \qquad \forall x \in \gamma\,.\end{equation}
As a consequence of (\ref{f:ineq}) and (\ref{f:equ}), we deduce that
$$u (x) = \phi _B (x) = \phi _ D (x) \qquad \forall x \in \gamma \,.$$
Namely, there holds
\begin{equation}\label{f:udiam}
u (x) =  c_0 \left[\rho_\Omega ^{4/3} - (\rho_\Omega - \dist_{\partial \Omega}(x))^{4/3}\right] \qquad \forall x \in \gamma \,.
\end{equation}
It follows from (\ref{f:udiam}) that the following relationship holds between the value of $|\nabla u|$ at the boundary points $y_\pm $ and the inradius:
$$|\nabla u (y_\pm) | = \frac{4}{3} c_0 \rho _\Omega ^ {1/3} = (3 \rho _\Omega) ^ {1/3} \qquad i  = 1, 2 \,.$$ 
Recalling that by assumption $u$ satisfies the Neumann condition $|\nabla u|(y) = a$ for all $y \in \partial \Omega$, 
we deduce that the value of the parameter $a$ is related to the inradius by the equality
\begin{equation}\label{f:valuea} 
a = (3 \rho _\Omega) ^ {1/3}\,.
\end{equation}
Thus, using \eqref{f:argmax} and (\ref{f:valuea}), we get
$$\mu = u (x_0) = c_0 \rho _\Omega ^ {4/3} = \frac{a ^ 4}{4}\,.$$
By Theorem \ref{t:ineqP}, this implies that the $P$-function associated with $u$ according 
to $(\ref{defP})$ is constant on $\overline \Omega$:
$$P (x) \equiv \frac{a ^ 4}{4} \qquad \forall x \in \overline \Omega \,.$$
By Proposition \ref{p:P1}, this implies that $u = \phi _\Omega$ (with $\rho _\Omega = a ^ 3/3$), and $\Cut (\Omega) = \high (\Omega)$.  \qed

\begin{figure}[ht]
\centering
\includegraphics{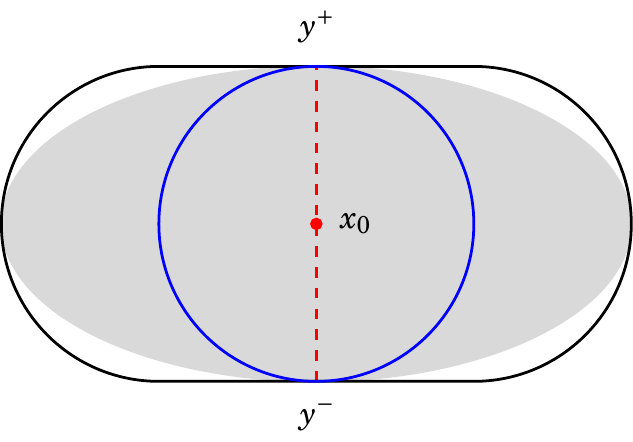}
\caption{A domain $\Omega$ (in gray) as in Theorem \ref{t:serrin2}, with $B \subset \Omega \subset D$.}   
\label{fig:ellipse}   
\end{figure}
\bigskip

\bigskip
By  combining Theorem \ref{t:serrin2} with the geometric results we obtained in \cite{CFb}, we can provide some geometric information 
on the shape of domains where the Serrin-type problem (\ref{f:serrin}) admits a solution, according to Corollary \ref{corgeo} below. 
We emphasize that  symmetry may hold or may fail according in particular to the boundary regularity of $\Omega$.

\bigskip
\begin{corollary}\label{corgeo} 

Under the same hypotheses of  Theorem \ref{t:serrin2}, we have:
\begin{itemize}
\item[(i)] if $\Omega$ is of class $C ^2$, then $\Omega$ is a ball;
\smallskip
\item[(ii)] if $n = 2$,   the set  $S:=\Cut (\Omega) = \high (\Omega)$ 
is a line segment (possibly degenerated into a point), and $\Omega$ is the stadium-like domain
$$\Omega=   \{ x \in \R ^2 \ :\ {\rm dist} (x, S) < \rho_\Omega \}\,.$$
 
\end{itemize}
\end{corollary}

\proof  Statement (i) follows directly from \cite[Thm.\ 12]{CFb}. Concerning statement
(ii), from 
\cite[Thm.\ 6]{CFb} we obtain
then the set  $S:=\Cut (\Omega) = \high (\Omega)$ 
is either a singleton or a $1$-dimensional manifold of class $C ^ {1,1}$, and $\Omega$ is the tubular neighborhood
$\Omega=  \{ x \in \R ^2 \ :\ {\rm dist} (x, S) < \rho_\Omega \}$.
Since we assumed that $\Omega$ is convex, the manifold $S$ must be necessarily a line segment, possibly degenerated into a point. 
\qed

\bigskip

\section{Regularity thresholds for the Dirichlet problem}\label{secreg}

As mentioned in the Introduction, it is well-known that for infinity--harmonic
functions one cannot expect $C^{1,1}$ regularity;
the expected regularity is in fact
$C^{1,\alpha}$ with $\alpha \leq 1/3$, which has been 
by now proved only in two space dimensions.

In this section we show that
a similar situation occurs also for
the solution $u$ to problem~\eqref{f:dirich}, specifically in view of its behavior near the set $K$ defined in \eqref{f:maxu}. 

We start with the following lemma, which allows to define the gradient flow associated with $u$ 
under the assumption that it is $C ^ {1,1}$ outside $K$.

\begin{lemma}\label{l:geo2}
Let $\Omega\subset\R^n$ be a nonempty open bounded set,
and let $u$ be the solution 
to problem \eqref{f:dirich}.
Assume that   
$u\in C^{1,1}({\Omega}\setminus K)$.
Then, for a.e.\ $x \in \Omega \setminus K$, 
there
exists a unique solution $\X(\cdot, x)$ to \eqref{f:geo},
defined in a interval $[0, T(x))$, where
$T(x)$ is defined by
\[
T(x) := \sup\{t\geq 0:\ \X(s,x)\in {\Omega}\setminus K\ \  \forall s\in [0,t] \}\,.
\]

Moreover, it holds $T (x) < + \infty$ and 
\begin{equation}\label{f:Treg}
\lim_{t\to T(x)^-} \X(t,x)  \in K.
\end{equation}
\end{lemma}

\proof
For every $x\in\Omega\setminus K$,
any solution of \eqref{f:geo} cannot
exit from $\{u \geq u(x)\}$ by the same argument
given in the proof of Lemma~\ref{l:geo}, and hence they are actually defined on $[0, + \infty)$. 
The uniqueness of the gradient flow associated with $u$
in ${\Omega}\setminus K$ follows from the local
Lipschitz regularity of $\nabla u$ assumed therein. 

Let us now prove \eqref{f:Treg}.  Recall that $\mu$ is defined according to \eqref{f:maxu}. 
We first prove the following 

\smallskip \textsl {Claim: 
There exists a  set $L \subseteq (0, \mu)$ with $|L| = \mu$ such that, for all $m \in L$, the condition \eqref{f:Treg} is satisfied for 
$\mathcal H ^ {n-1} \hbox{-a.e.}\ x \in \{ u = m \}$.} 

\smallskip

Let us define $L$ as the set of values $m \in (0, \mu)$ such that  $u$ is twice differentiable $\mathcal H ^ {n-1}\hbox{a.e.}$ on $\{ u = m\}$.
By the coarea formula, if $Z$ is the set of points in $\Omega \setminus K$ where $u$ is not twice differentiable,   we have 
\begin{equation}\label{f:coarea}
0 = \int _{Z}  |\nabla u| \, dx = \int _0 ^ {\mu} \, dm \int _{\{ u  = m \} \cap Z } \, d \mathcal H ^ {n-1} (y) \,.
\end{equation}
We observe that $| \nabla u| $ remains strictly positive $\mathcal L ^n$-a.e.\ in $\Omega \setminus K$;
otherwise, since $u$ is twice differentiable $\mathcal L ^n$-a.e.\ in $\Omega \setminus K$, the pde $\Delta _\infty u = -1$ would not be satisfied. 
Hence we infer from \eqref{f:coarea} that, for ${\mathcal L}^1$-a.e. $m \in (0, \mu)$, the set $\{ u  = m \}\cap Z$ is $\mathcal H ^ {n-1}$-negligible, so that $L$ is of full measure in $(0, \mu)$.

From now on, let $m$ denote a fixed value in $L$. 
For $x \in \{ u = m \}$, set
\begin{equation}\label{f:N}
N (x) :=  \Big \{ t \in [0, T(x)] \ :\ u   \text{ is not twice differentiable at }  \X (t, x) \Big \}\,.
\end{equation}

By repeating the arguments given in the proof of
Lemma~\ref{l:approx2}, we obtain that
$\mathcal{L}^1(N(x)) = 0$ for
$\mathcal{H}^{n-1}$-a.e.\ $x \in \{ u = m \}$. 

Let us prove that \eqref{f:Treg} holds
for every $x_0\in \{ u = m \}$ such that both the conditions
$\mathcal{L}^1(N(x_0)) = 0$ and $u$ twice differentiable at $x_0$ hold. 

Let $x_0$ be such a point, and  let
$p(x_0) := \lim _{t \to T  (x_0)^-}\X (t, x_0)$
(observe that such limit exists since
$\frac{d}{dt} \X (t, x_0)$ is bounded).

Since $u$ is twice differentiable at $x_0$, it cannot be $\nabla u (x_0) = 0 $; otherwise, as already noticed above, the pde $-\Delta _\infty u = 1$ would not be satisfied. 

Then, by the very definition of $T(x_0)$,
in order to prove \eqref{f:Treg} is is enough
to show that $T(x_0) < +\infty$. Indeed in this case we have that $p (x_0) \in \partial \Omega \cup K$, but the possibility that $p (x_0) \in \partial \Omega$ is excluded by the fact that $u$ increases along the flow. 

Let us show that $T (x_0 ) < + \infty$. 
Let
\[
\gamma(t) := \X(t, x_0), \quad
\varphi(t) := u(\gamma(t)) \,,\qquad
t\in [0, T), \ T := T(x_0).
\]
Since $\mathcal{L}^1(N(x_0)) = 0$, the $P$-function is constant along $\gamma$. Then by Proposition \ref{p:P2}  for some $\lambda >m$ we have 
$\varphi (t) = \overline \varphi (t)$ for every $t \in [0, T)$, with 
\[
\overline{\varphi}(t) :=
\begin{cases}
\lambda - (\sqrt{\lambda-m} - t)^2,
&\text{if}\ t\in [0, \sqrt{\lambda-m}),\\
\lambda,
& \text{if}\ t\geq \sqrt{\lambda-m}\,.
\end{cases}
\] 

Let us show that $T = \sqrt{\lambda-m}$.
It is clear that $T\geq \sqrt{\lambda-m}$,
since $\dot\varphi(t) \neq 0$ for $t\in [0, \sqrt{\lambda-m})$,
so that $\gamma(t)\not\in K$ for $t\in [0, \sqrt{\lambda-m})$
because $\nabla u = 0$ on $K$
(recall that, by \cite{Lind}, $u$ is differentiable everywhere in $\Omega$).
On the other hand, the trajectory $\gamma$
enters in finite time $\sqrt{\lambda -m}$ in a point $p$
where $\nabla u(p) = 0$, which cannot happen
if $\nabla u$ is locally Lipschitz continuous in
a neighborhood of $p$ (since otherwise uniqueness would be contradicted). 
Hence $p \in K$, $\lambda = \mu$ 
and $T(x_0) = \sqrt{\lambda-m}$.
This concludes the proof of the claim. 

Finally, let us prove that the statement of the lemma follows from the claim. Let $F$ denote the set of points $x \in \Omega \setminus K$ such that \eqref{f:Treg} is false. Note that the complement of $F$, namely the set where \eqref{f:Treg} holds,  is closed by continuous dependence on initial data. In particular, this ensures that $F$ is $\mathcal L^n$-measurable. 
Then, by the claim and the coarea formula, we have
\begin{equation}\label{f:coarea2}
0 = \int _0 ^ {\mu} \, dm \int _{\{ u  = m \} \cap F } \, d \mathcal H ^ {n-1} (y) = \int _{F}  |\nabla u| \, dx \,.
\end{equation}
We recall that, since by assumption $u \in C ^ {1,1} (\Omega \setminus K)$, $u$ is twice differentiable 
$\mathcal L^n$-a.e.\ on $\Omega \setminus K$,  and hence $|\nabla u|>  0$ 
$\mathcal L^n$-a.e.\ on $\Omega \setminus K$ (because  the pde $\Delta _\infty u = -1$ is not fulfilled at points where $\nabla u = 0$).   Then by \eqref{f:coarea2} we deduce that $|F| = 0$. \qed

\begin{corollary}\label{c:P11} 
Let $\Omega\subset\R^n$ be a nonempty open bounded set,
and let $u$ be the solution 
to problem \eqref{f:dirich}.
Assume that   
$u\in C^{1,1}({\Omega}\setminus K)$.
Then
$$P (x) = \mu \qquad \forall x \in \Omega \setminus K \,.$$
In particular, if $\ipu$ holds, we have 
\begin{equation}\label{f:bordo}
\frac{|\nabla u|^ 4(y) }{4} = \mu \qquad \forall y  \in \partial \Omega\,.
\end{equation}
\end{corollary} 

\proof Since by assumption $P$ is continuous on $\Omega \setminus K$, it is enough to show that the equality $P (x) = \mu$ holds almost everywhere on $\Omega \setminus K$. Namely, let us show that it holds for every $x \in \Omega \setminus K$ such that \eqref{f:Treg} holds and $\mathcal L ^ 1 (N(x)) = 0$. (Actually, both these conditions are satisfied up to a $\mathcal L^n$-negligible set, by the same arguments used in the proof of Lemma \ref{l:geo2}). Let $x \in \Omega \setminus K$ be such that \eqref{f:Treg} holds and $\mathcal L ^ 1 (N(x)) = 0$. 
Set $\gamma (t):= \X (t, x)$, for $t \in [0, T(x))$. Since $\mathcal L ^ 1 (N(x)) = 0$, $P$ is contant along $\gamma$ and, since \eqref{f:Treg} holds, we have $P (\gamma (t)) = \mu$ on $[0, T(x))$. In particular, $P (x) = \mu$. 
Finally, under assumption $\ipu$,  the equality \eqref{f:bordo} follows immediately by combining the equality $P (x) \equiv \mu$ holding on $\Omega \setminus K$ with the Dirichlet condition $u = 0$ satisfied on $\partial \Omega$. \qed

\bigskip

\begin{proposition}\label{p:notreg1}
Assume that $\Omega$ satisfy the following conditions: $\ipo$, $\Cut (\Omega ) \neq \high (\Omega)$, and 
there exists an inner  ball $B$ of radius $\rho _\Omega$ which meets $\partial \Omega$ at two points
lying on the same diameter of $B$. 
Further, assume that the unique solution to problem \eqref{f:dirich} satisfies $\ipu$. 
Then $ u \not \in C^ {1,1} (\Omega \setminus K)$. 
\end{proposition}

\proof  Assume by contradiction that $u\in C^ {1,1} (\Omega \setminus K)$. Then, by assumption $\ipu$ and Corollary 
\ref{c:P11}, we have that \eqref{f:bordo} holds. Hence, $u$ is a solution to the overdetermined boundary value problem $(\ref{f:serrin})$. Since we have assumed also $\ipo$ and the existence of an inner ball $B$ which meets $\partial \Omega$ at two diametral points, by Theorem \ref{t:serrin2} we infer that
$\Cut (\Omega ) = \high (\Omega)$, contradiction. \qed 

\bigskip
The assumptions made on $\Omega$ in the above proposition are satisfied for instance when $\Omega$ is an ellipse. 
For general domains, assuming 
that the solution $u$ to problem \eqref{f:dirich}  is $C ^{1,1}$ near $K$, we obtain the following result
which gives an indication that the optimal expected regularity of $u$ (up to $K$) 
is $C ^ {1,1/3}$.

\begin{proposition}\label{p:notreg2}
Let $u$ be the solution
to problem \eqref{f:dirich}, and let $A$ be a neighborhood of $K$.

Assume that $u\in C^{1,1}(A \setminus K)$.
Then for every $\alpha > 1/3$ it holds $u \not \in C ^ {1, \alpha} (A)$. 
\end{proposition}

\begin{proof}  
Assume that $u\in C^{1,\alpha}(A)$, 
i.e.\ $u\in C^1(A)$ and there exists $C>0$ such that
\[
|\nabla u (x) - \nabla u(y)|
\leq C |x-y|^{\alpha}
\qquad \forall x,y\in A\, .
\] 
We are going to show that necessarily it must be $\alpha \leq 1/3$.

Let us choose $m$ such that
$E:=\{x\in\Omega:\ u(x) \geq  m \} \subset A $.

Since $u \in C ^ {1,1} (E \setminus K)$ and $u$ is differentiable on $K$, we can associate with the restriction of $u$ to $E$ the gradient flow $\X$ according to Lemma 
\ref{l:geo2}. 

Let $L \subseteq (0, \mu)$ be as in the Claim contained in the proof of Lemma \ref{l:geo2}. 
Let $m' \in L \cap (0, m)$ be such that \eqref{f:Treg} holds $\mathcal H  ^ {n-1}$-a.e.\ on $\{ u = m'\}$. 
By repeating the arguments given in the proof of
Lemma~\ref{l:approx2}, we obtain that
$\mathcal{L}^1(N(x)) = 0$ for
$\mathcal{H}^{n-1}$-a.e.\ $x \in \{ u = m' \}$, with $N (x)$ defined as in \eqref{f:N}. 

Then we can pick  
 $x_0 \in \{ u = m' \}$ be such that $\mathcal L ^ 1 (N (x_0)) = 0$ and  \eqref{f:Treg} holds at $x_0$. Set
\[
\gamma(t) := \X(t, x_0), \quad
\varphi(t) := u(\gamma(t)),\qquad
t\in [0, T), \ T := T(x_0).
\]
Since $\mathcal L ^ 1 (N (x_0)) = 0$, the $P$-function is constant along $\gamma$. Moreover, since  \eqref{f:Treg} holds at $x_0$, the value of the constant is equal to $\mu$, namely it holds 
\[
P(\gamma(t)) =
\frac{1}{4} |\nabla u(\gamma(t))|^4 + u(\gamma(t))
= \mu 
\qquad\forall t\in [0, T). 
\] 
Then, by Proposition \ref{p:P2}, we have $\varphi (t) = \overline \varphi (t)$ for every $t \in [0, T)$, with 
\[
\overline{\varphi}(t) :=
\begin{cases}
\mu - (\sqrt{\mu-m} - t)^2,
&\text{if}\ t\in [0, \sqrt{\mu-m}),\\
\mu,
& \text{if}\ t\geq \sqrt{\mu-m}\,.
\end{cases}
\] 
We recall that the trajectory $\gamma$ cannot
reach a maximum point of $u$ 
in a time $t < T$,
whereas it approaches $K$ as $t\to T^-$,
i.e.
\begin{equation}\label{f:sottomax}
\varphi(t) < \mu
\quad\forall t\in [0, T)\,,
\qquad
\lim_{t\to T^-} \varphi(t) \in K\,.
\end{equation}
Thus we deduce that
$T$ is finite and
\begin{equation}\label{f:varphi}
T = \sqrt{\mu-m},\qquad
\varphi(t) = \mu - (T- t)^2
\quad \forall t\in [0, T].
\end{equation}
(For later convenience we have extended $\varphi$
up to time $T$.)

Incidentally, notice that the finitenes of $T$ shown in \eqref{f:varphi} already implies that $u \not \in C ^ {1,1} (A)$ (otherwise 
by uniqueness it should be $T = + \infty$).  
For every $t\in [0, T]$ we have 
\begin{equation}
\label{f:stimaalpha}
\begin{split}
\mu - \varphi(t) & =
\varphi(T) - \varphi(t)
= \int_t^T |\nabla u(\gamma(s))|^2\, ds
\\ & =
\int_t^T |\nabla u(\gamma(s)) - \nabla u(\gamma(T))|^2\, ds
\leq
C^2 \int_t^T |\gamma(s)-\gamma(T)|^{2\alpha}\, ds\,.
\end{split}
\end{equation}
In order to estimate the last integral in~\eqref{f:stimaalpha},
let us consider the auxiliary function
$z(t) := |\gamma(t)-\gamma(T)|$, $t\in [0, T]$.
Since $z(t) > 0$ for every $t\in [0, T)$,
for such values of $t$ we have 
\[
\begin{split}
\dot{z}(t) & = \dot{\gamma}(t) \cdot
\frac{\gamma(t)-\gamma(T)}{|\gamma(t) - \gamma(T)|}
\geq - |\dot{\gamma}(t)|
= - |\nabla u(\gamma(t))|
\\ & = - |\nabla u(\gamma(t)) - \nabla u(\gamma(T))|
\geq - C|\gamma(t) - \gamma(T)|^{\alpha}
= -C \, z(t)^{\alpha}\,.
\end{split}
\]
Since the maximal solution in $[0,T]$ of the
Cauchy problem
$\dot{z} = - C\, z^{\alpha}$, $z(T) = 0$, is 
\[
\overline{z}(t) :=
\left[
C(1-\alpha)(T-t)\right]^{1/(1-\alpha)}
,
\qquad
t\in [0, T],
\]
we conclude that
\begin{equation}
\label{f:disd}
|\gamma(t)-\gamma(T)| = z(t) \leq
\overline{z}(t)
= C_1\, (T-t)^{1/(1-\alpha)}, \qquad
\forall t\in [0, T]\,,
\end{equation}
where 
$C_1:= \left[C(1-\alpha)\right]^{1/(1-\alpha)}$.
By  \eqref{f:stimaalpha} and \eqref{f:disd},
we deduce that
\[
\mu-\varphi(t) \leq C_2 \int_t^T (T-s)^{2\alpha/(1-\alpha)}\, ds
= C_2 (T-t)^{(1+\alpha)/(1-\alpha)}\,,
\qquad\forall t\in [0,T],
\]
where $C_2 := C^2\, C_1^{2\alpha}$.
Taking into account the explicit form of $\varphi$
given in~\eqref{f:varphi}
we obtain 
\[
(T- t)^2 \leq
C_2 (T-t)^{(1+\alpha)/(1-\alpha)}\,,
\qquad \forall t\in [0,T],
\]
that clearly cannot be satisfied if $2 < (1+\alpha)/(1-\alpha)$,
i.e.\ if $\alpha > 1/3$.
\end{proof}

\bigskip

{\bf Acknowledgments.} We gratefully acknowledge Filippo Gazzola and Bernd Kawohl for sharing some interesting discussions
about the argument used in Proposition \ref{p:P1}.


\def\cprime{$'$}

\end{document}